\documentclass{amsart}
\usepackage[utf8]{inputenc}

\usepackage{amssymb}		
\usepackage{graphicx}		
\usepackage{hyperref}		

\usepackage{amsmath}
\usepackage{amsthm}
\usepackage{bbm}
\usepackage{enumerate}
\usepackage{tikz}
\usepackage{tikz-cd}
\usetikzlibrary{hobby}

\theoremstyle{plain}
\newtheorem{theorem}{Theorem}[section]
\newtheorem{corollary}[theorem]{Corollary}
\newtheorem{lemma}[theorem]{Lemma}

\theoremstyle{definition}
\newtheorem{definition}[theorem]{Definition}
\newtheorem{remark}[theorem]{Remark}
\newtheorem{example}[theorem]{Example}

\newcommand{\C}{\mathbb{C}}

\newcommand{\Z}{\mathbb{Z}}
\newcommand{\N}{\mathbb{N}}
\newcommand{\lc}{\text{lc}}
\newcommand{\lcu}{\text{lcu}}
\newcommand{\lcs}{\text{lcs}}
\newcommand{\sync}{\text{sync}}
\newcommand{\I}{\mathcal{I}}

\newcommand{\lang}{\mathcal{L}}
\newcommand{\ev}{{\text{even}}}

\newcommand{\parens}[1]{\left( #1 \right)}
\newcommand{\rank}{{\text{rank}}}
\newcommand{\word}[1]{\mathtt{#1}}
\newcommand{\ie}{\textit{i.e.}\ }

\title[Synchronizing Dynamical Systems: Shift Spaces and $K$-Theory]{Synchronizing Dynamical Systems: \\ Shift Spaces and $K$-Theory}
\author{Robin J. Deeley}
\address{Robin J. Deeley,   Department of Mathematics,
University of Colorado Boulder
Campus Box 395,
Boulder, CO 80309-0395, USA }
\email{robin.deeley@colorado.edu}
\author{Andrew M. Stocker}
\address{Andrew M. Stocker,   Department of Mathematics,
University of Colorado Boulder
Campus Box 395,
Boulder, CO 80309-0395, USA }
\email{andrew.stocker@colorado.edu}
\thanks{This work was partially supported by the National Science Foundation under Grants No. DMS 2000057 and 2247424 and Simons Foundation Gift MP-TSM-00002896.}

\begin{document}

\begin{abstract}
Building on our previous work, we give a thorough presentation of the techniques developed for synchronizing dynamical systems in the special case of synchronizing shift spaces.  Following work of Thomsen, we give a construction of the homoclinic, the heteroclinic, and synchronizing heteroclinic $C^\ast$-algebras along with the synchronizing ideal of a shift space in terms of Bratteli diagrams. The algebras introduced in our previous work (the synchronizing ideal, and synchronizing heteroclinic algebra) are discussed in detail. In the sofic shift case, these algebras are shown to be related to the $C^\ast$-algebras of its minimal left and minimal right presentations. Several specific examples are discussed to demonstrate these techniques.  For the even shift we give a complete computation of all the associated invariants. We discuss these algebras for a sofic shift that is not of almost finite type and for a number of strictly non-sofic synchronizing shifts.  In particular we discuss the rank of the $K$-theory of the homoclinic algebra of a shift space and its synchronizing ideal and its implications. We also give a construction for producing from any minimal shift a synchronizing shift whose set of non-synchronizing points is exactly the original minimal shift.
\end{abstract}

\maketitle

\section*{Introduction}

In \cite{deeley2022} we have developed the basic theory of synchronizing dynamical systems (or henceforth just ``synchronizing systems'') and building on work of Thomsen the $C^\ast$-algebras associated with such systems.  The theory of synchronizing systems is fundamentally based on techniques established in the study of Smale spaces.  Smale spaces are dynamical systems which are uniformly hyperbolic \cite{putnam99, putnam_1996, ruelle_2004, smale67}.  Synchronizing systems, on the other hand, are hyperbolic almost everywhere. That is, there is an open dense set of points that have local hyperbolic behavior. The precise definition of having ``local hyperbolic behavior" is involved. We will discuss this later in the introduction and in Section \ref{sec:sync-words}. It is also discussed in greater detail in \cite{deeley2022}.  

The term ``synchronizing'' is borrowed directly from the field of symbolic dynamics where there is a notion of a synchronizing shift \cite{blanchard86, fiebig91}.  Our definition of synchronizing system generalizes this class of shift spaces, see Lemma \ref{lem:same-def}.  Synchronizing shifts are thus an important example of synchronizing systems. In fact, the focus of the current paper is on synchronizing shifts, their $C^\ast$-algebras, and the $K$-theory of these $C^\ast$-algebras viewed as an invariant of the shift space.

Synchronizing systems are expansive. Expansive dynamical systems are topological dynamical systems that exhibit sensitivity to initial conditions. The precise definition is as follows: we say that a dynamical system $(X,\varphi)$ (throughout $X$ is a compact metric space and $\varphi : X \to X$ is a homeomorphism) is \emph{expansive} if there exists a constant $\varepsilon_X > 0$ such that $d(\varphi^n(x),\varphi^n(y)) \leq \varepsilon_X$ for all $n \in \Z$ implies $x = y$. Shift spaces are expansive and more generally of crucial importance in the study of expansive systems. In particular, every expansive system is a factor of a shift space \cite{walters1981}. Hence, there is a direct connection between the combinatorial nature of shift spaces and the dynamics of expansive systems.  This further justifies the current paper's focus on shift spaces in particular.  Additionally, we 

\begin{enumerate}
\item generalize results from the class of shifts of finite type where the $C^\ast$-algebras and their $K$-theory are well-understood,
\item provide explicit computations of the invariants presented in \cite{deeley2022} building on work of Thomsen, and
\item develop examples showing the limits of possible generalizations.
\end{enumerate}

A synchronizing shift is an irreducible shift space $X$ which admits a word $w \in \lang(X)$ with the following property: if $u,v \in \lang(X)$ are such that $uw \in \lang(X)$ and $wv \in \lang(X)$, then $uwv \in \lang(X)$.  An arbitrary shift space may not even have a synchronizing word, and not all words in a synchronizing shift have this property (a word with this property is itself called synchronizing).  A synchronizing shift in which all words are eventually synchronizing is a shift of finite type, see \cite[Theorem 2.1.8]{lindmarcus}.  However the class of synchronizing shifts contains many shift spaces that are not shifts of finite type.  Most importantly every irreducible sofic shift is synchronizing, see \cite[Corollary 3.11]{deeley2022}. Beyond the irreducible sofic shifts there are many interesting synchronizing shifts that are strictly non-sofic, see Section \ref{sec:infinite-rank} and Section \ref{sec:minimal} for examples.

We will now draw a connection between the concept of a synchronizing word and a point in a shift space having local hyperbolic behavior.  Suppose $X$ is a synchronizing shift space and $x \in X$, $N \geq 0$ are such that $x_{[-N,N]}$ is a synchronizing word.  If  $y^-$ and $z^+$ are any left infinite and right infinite rays that may be concatenated to the left and right of $x_{[-N,N]}$ (respectively), then $x' = y^- x_{[-N,N]} z^+$ is a valid element of $X$ due to $x_{[-N,N]}$ being synchronizing.  Hence the set of such left/right-infinite rays $\{ (y^-, z^+) \}$ is a Cartesian product.  In fact this set corresponds to a set of pairs $(y,z)$ in $X$ for which $z$ is stably equivalent to $x$, in the sense that for some $N \in \Z$ we have $z_n = x_n$ for all $n \geq N$; and $y$ is unstably equivalent to $x$ in the symmetric sense.  By the above, the set of such $(y,z)$ is also a Cartesian product.  Stable and unstable equivalence are a dynamical notion, and it is in this sense that the dynamical behavior of $(X,\sigma)$ near $x$ can be seen to be hyperbolic.  This connection is spelled out explicitly in Section \ref{sec:sync-words}.

With techniques developed here and in \cite{deeley2022}, synchronizing systems represent a significant broadening of the types of expansive dynamical systems for which we have tools to study their $C^\ast$-algebras and $K$-theory.  In \cite{thomsen2010c}, Thomsen has shown that a $C^\ast$-algebra $A(X,\varphi)$ called the \emph{homoclinic algebra} can be associated to any expansive dynamical system.  See Section \ref{sec:shift-bratteli-homoclinic} for a description of this algebra when $X$ is a shift space --- this construction is entirely due to Thomsen.  In the case when $X$ is a synchronizing system however, we have shown in \cite{deeley2022} that there is an ideal $\I_\sync(X,\varphi) \subseteq A(X,\phi)$ determined by the set of synchronizing points in $X$.  Hence the fundamental structure of the homoclinic algebra of a synchronizing system can be understood in terms of the following short exact sequence.
\[ 0 \longrightarrow \I_\sync(X, \varphi) \longrightarrow A(X,\varphi) \longrightarrow A(X,\varphi) / \I_\sync(X,\varphi) \longrightarrow 0 \]
When $X$ is a synchronizing shift, the groupoid $G^\lc(X,\sigma)$ used to build the homoclinic algebra is amenable and so we obtain that \[ A(X,\sigma) / \I_\sync(X,\sigma) \cong C^\ast\left(G^\lc(X,\sigma)|_{X_\text{non-sync}}\right) \] in the above short exact sequence.  Note that in general this quotient is not the same as the homoclinic algebra of the subshift of non-synchronizing points $X_\text{non-sync}$ --- see Section \ref{sec:gicar} for an explicit example.

Irreducible finitely presented systems, introduced by Fried in \cite{fried1987}, are shown to also be synchronizing systems in \cite{deeley2022}.  The shift spaces that are finitely presented systems are exactly the sofic shifts introduced by Weiss in \cite{Weiss1973SubshiftsOF}.  Hence an irreducible sofic shift is synchronizing.  However, when a sofic shift is mixing we may understand the structure of $\I_\sync(X,\sigma)$ in a very explicit way as follows.  In \cite{thomsen2010c}, Thomsen introduced the heteroclinic algebras for an expansive dynamical system in which periodic points are dense.  A main result of \cite{deeley2022} is that synchronizing systems, and thus also irreducible finitely presented systems, have a dense set of periodic points.  However, Thomsen's heteroclinic algebras are built using the set of all periodic points in $X$, and in a synchronizing system the synchronizing and non-synchronizing periodic points behave differently.  In \cite{putnam99}, Putnam and Spielberg construct the stable and unstable algebras of a Smale space by fixing a periodic point (or finite set of periodic orbits).  They then show that up to Morita equivalence this choice does not matter.  Inspired by this, we modify Thomsen's heteroclinic algebras so that we instead fix a single \emph{synchronizing} periodic point (or more generally a finite set of them).  We have also shown that this choice does not matter up to Morita equivalence, see \cite[Theorem 5.5]{deeley2022}.  Hence we call this modified construction collectively the heteroclinic synchronizing algebras of $(X, \varphi)$.  More specifically we associate to a synchronizing system the stable synchronizing algebra $S(X,\varphi,p)$ and unstable synchronizing algebra $U(X,\varphi,p)$ where $p \in X$ is a synchronizing periodic point.  A fundamental result in \cite{deeley2022} is that when $(X,\varphi)$ is a mixing finitely presented system, we have the Morita equivalence \[ S(X,\varphi,p) \otimes U(X,\varphi,p) \sim_\text{M.E.} \I_\sync(X,\varphi) \,. \]
Hence this Morita equivalence in particular holds for mixing sofic shifts, of which the even shift is a prime example.  In Section \ref{sec:even-shift} we show how this Morita equivalence can be seen explicitly in our computations. 

Furthermore, we prove that the $C^\ast$-algebra $S(X,\varphi,p)$ is isomorphic to the stable algebra of the minimal right resolving cover of the given sofic shift. The minimal right resolving cover is a shift of finite type and its stable algebra is well-understood. In a similar way, $U(X, \varphi, p)$ is isomorphic to the unstable algebra of the minimal left resolving cover of the given sofic shift. In summary, one of the main results of this paper is a complete description of the $C^\ast$-algebras associated to a mixing sofic shift in terms of its minimal left-resolving presentation, its minimal right-resolving presentation, and its set of non-synchronizing points (that is, points without a local hyperbolic structure). For strictly non-sofic shifts the situation is more complex, see Section \ref{sec:infinite-rank} and Section \ref{sec:minimal}.

Although synchronizing shift spaces are a special case of the synchronizing systems discussed in \cite{deeley2022}, we aim to make this paper as self-contained as possible.  In Section \ref{sec:shift-spaces} we will give an introduction to shift spaces.  We will also introduce shifts of finite type and sofic shifts.  Section \ref{sec:sync-words} will be specifically focused on synchronizing shifts.  In this section we show that our definition of synchronizing given in \cite{deeley2022} generalizes the class of synchronizing shifts.

In Section \ref{sec:local-conjgacy} we discuss the equivalence relations defined in \cite{deeley2022} in the context of shift spaces.  Then in Section \ref{sec:shift-algebras} we show how these equivalence relations can be produced inductively using Bratteli diagrams.  From these Bratteli diagrams we define several different $C^\ast$-algebras.  The $K$-theory of these $C^\ast$-algebras are conjugacy invariants of the underlying shift spaces \cite{thomsen2010c}.  In the examples to follow we show how these $K$-theory groups contain information about the asymptotic behavior of shift spaces.  We also discuss the short exact sequence determined by the synchronizing ideal, a Morita equivalence which relates these $C^\ast$-algebras, and the relationship between the (un)stable synchronizing algebra and those of the minimal right and minimal left resolving presentations.

In Section \ref{sec:even-shift} through Section \ref{sec:minimal} contains several examples. They demonstrate some of the possible behavior (both at the shift level and the $C^\ast$-algebra level) that can be observed in a synchronizing shift. We give examples for which
\begin{enumerate}[(i)]
    \item the set of non-synchronizing points is a single point (Section \ref{sec:even-shift}),
    \item the set of non-synchronizing points is infinite (Section \ref{sec:dc-shift}),
    \item a sofic shift whose synchronizing ideal is not Morita equivalent to one of a shift of finite type (Section \ref{sec:sofic-not-finite}),
    \item\label{item:finite-rank} the rank of the $K$-theory of the homoclinic algebra is infinite (Section \ref{context-free-shift}),
    \item the quotient algebra can be described intrinsically (Section \ref{sec:gicar})
    \item and lastly an example where the set of non-synchronizing points is a minimal shift space (Section \ref{sec:minimal}).
\end{enumerate}
The even shift, denoted $X_\ev$, is presented in Section \ref{sec:even-shift}.  This is intended to be a complete example where all invariants are computed.  We remark that Thomsen has computed the Bratteli diagram for the homoclinic algebra of the even shift in \cite{thomsen2010c}, although our computation is perhaps more explicit. The even shift has a single non-synchronizing point, and thus we obtain the short exact sequence \[ 0 \longrightarrow \I_\sync(X_\ev, \sigma) \longrightarrow A(X_\ev, \sigma) \longrightarrow \C \longrightarrow 0 \, \]
where $\I_\sync(X_\ev, \sigma)$ is completely understood using the golden mean shifts (which is both the minimal right and left resolving cover of the even shift). We also discuss how the failure of the even shift to be a shift of finite type can be seen through the fact that the homoclinic algebra of the even shift is not Morita equivalent to the tensor product of its heteroclinic algebras.

However it is not the case the set of non-synchronizing points must be finite or even countable --- such an example is demonstrated in the case of the charge-constrained shift (which is of almost finite type) in Section \ref{sec:dc-shift}.  Also note for the shift space in Section \ref{sec:gicar}, the subshift of non-synchronizing points is exactly the full 2-shift.

The significance of (\ref{item:finite-rank}) is that it is known that $K$-theory of the homoclinic algebra of a sofic shift space is of finite rank, see Section \ref{sec:infinite-rank}.  Hence, the rank of the $K$-theory of the homoclinic algebra of a shift space $X$ may be seen as obstructing $X$ belonging to the class of sofic shifts.  In Section \ref{context-free-shift}, we give a thorough examination of an example of a synchronizing shift for which the rank of the $K$-theory of the homoclinic algebra is infinite.  This involves using infinite labeled graph representations and relating the rank of the $K$-theory of the homoclinic algebra to paths on these graphs.   The authors suspect that a similar technique could be use to prove that the $K$-theory of the homoclinic algebra is infinite rank for other synchronizing shifts which have the following property: there exists a finite set $F \subseteq \lang_\sync(X)$ such that $w \in \lang(X)$ is synchronizing if and only if $u \sqsubseteq w$ for some $u \in F$.

One may wonder what kinds of shift spaces may appear as the subshift of non-synchronizing points in a synchronizing shift.  We have seen an example where the subshift of non-synchronizing points is the full 2-shift, and the shift space presented in Section \ref{context-free-shift} also has a shift of finite type as its subshift of non-synchronizing points.  Unfortunately, there does not appear to be any sort of constraint on the type of shifts that can appear.  In Section \ref{sec:minimal} we show that, given any minimal shift $M$, we can construct a synchronizing shift whose set of non-synchronizing points is exactly $M$.

There are number of $C^\ast$-algebras associated to a shift space in this paper. We summarize these algebras here to avoid any confusion. Three were defined by Thomsen \cite{thomsen2010c} (building on work of Ruelle \cite{ruelle1988} and Putnam \cite{putnam_1996}). Thomsen's constructions are the start point for our work. Thomsen's defined the homoclinic and the two heteroclinic (one for the forward shift and one for the backward shift) $C^\ast$-algebras. In our previous work \cite{deeley2022}, we defined three $C^\ast$-algebras. They are the synchronizing ideal and the two synchronizing heteroclinic algebras. The key difference between our algebras and Thomsen's is that ours avoid the non-synchronizing points. The results of the present paper show how by excluding the non-synchronizing points, we get $C^\ast$-algebras that are either directly or at least more closely related to the $C^\ast$-algebras associated to shifts of finite type (and in higher dimensions Smale spaces).

\section*{Acknowledgments}
We thank Ian Putnam for many helpful comments.

\section{Shift Spaces}\label{sec:shift-spaces}

Let $\mathcal{A}$ be a finite set of symbols which we will call an \emph{alphabet}.  Then the space \[ \mathcal{A}^\mathbb{Z} = \{ \ldots x_{-2} x_{-1} . x_0 x_1 x_2 \ldots \mid x_i \in \mathcal{A} \text{ for all } i \in \Z \} \] with the product topology is a compact topological space which is also totally disconnected.  We will typically write a sequence with a decimal point to indicate the zero-th index.  Let $\sigma : \mathcal{A}^\mathbb{Z} \to \mathcal{A}^\mathbb{Z}$ be the map defined by $\sigma(x)_i = x_{i+1}$.  This is called the \emph{shift map}.  It is a homeomorphism of $\mathcal{A}^\mathbb{Z}$ \cite{putnamNotes19}.

\begin{definition} 
A \emph{shift space} is a dynamical system $(X,\sigma)$ where $\sigma$ is the shift map and $X \subseteq \mathcal{A}^\Z$ is a closed subspace of $\mathcal{A}^\mathbb{Z}$ which is invariant under $\sigma$.  
\end{definition}

Since the dynamical map on a shift space is always $\sigma$ (implicitly restricted to $X$) we will often denote a shift space $(X, \sigma)$ as just $X$.  For a more detailed treatment of shift spaces, see \cite{lindmarcus}.

It will be useful for us to consider all the words of finite length (including the empty word) that appear in elements of a shift space.  Let the \emph{language} of $X$, denoted $\mathcal{L}(X)$, be the collection of all finite strings that appear in elements of $X$, that is, $w \in \mathcal{L}(X)$ if and only if $w = x_k x_{k+1} \cdots x_{k+n}$ for some $x \in X$ --- we denote this $w \sqsubseteq x$.  We will also often denote the word $x_i x_{i+1} \cdots x_j$ as $x_{[i, j]}$ for $-\infty < i \leq j < \infty$, or $x_{[i,j)}$ for the word $x_i x_{i+1} \cdots x_{j-1}$ for $-\infty < i < j \leq \infty$ (likewise for $x_{(i,j]}$).  Additionally we will also denote the length of a word $w$ as $|w|$.  Lastly, we will write the set of words of length $n$ appearing in elements of $X$ as $\mathcal{L}_n(X)$, that is \[ \mathcal{L}_n(X) = \mathcal{L}(X) \cap \mathcal{A}^n = \{ w \in \lang(X) \mid |w| = n \} \,. \]

\begin{example}
For an example of what $\mathcal{L}(X)$ looks like we can consider the \emph{full 2-shift}, which is defined as the shift space $X = \{\word{0}, \word{1}\}^\mathbb{Z}$.  Then $\displaystyle \mathcal{L}(X) = \bigcup_{n \geq 0} \{\word{0}, \word{1}\}^n$ which is the set of all finite binary words.
\end{example}

The topology on $X$ is generated by what are called cylinder sets.  A \emph{cylinder set} is a set of the form \[ C(w) = \{ x \in X \mid x_{[-n,n]} = w \} \] for $w \in \mathcal{L}_{2n + 1}(X)$.  Note that these are clopen sets in the topology on $X$.  Hence if $x \in X$ then the cylinder sets $C(x_{[-n,n]})$ are neighborhoods of $x$ for all $n \geq 0$.  Furthermore, the topology on a shift space is given by the following metric: \[ d(x,y) = \,\inf_{n \geq 1} \left\{ 2^{-|n|} \,\middle|\, x_i = y_i \text{ for all } |i| < n \right\} \] when $x_0 = y_0$, and $d(x,y) = 1$ otherwise.

In fact it is not too hard to see that shift spaces are also expansive.  The proof is as follows, let $X$ be a shift space and fix $\varepsilon_X = 1$.  Observe that for two distinct points $x,y \in X$ there is some $n \in \mathbb{Z}$ such that $x_n \neq y_n$, meaning $d(\sigma^n(x), \sigma^n(y)) = 1$.  This shows that $\varepsilon_X = 1$ is an expansiveness constant for any shift space assuming we are using the above metric.

\begin{definition}
A shift space $X$ is called \emph{irreducible} if for each ordered pair $w,v \in \lang(X)$, there is a $u \in \lang(X)$ such that $wuv \in \lang(X)$.
\end{definition}

\begin{definition}
A shift space $X$ is called \emph{mixing} if each each $w,v \in \lang(X)$ there is an $N$ such that for each $n \geq N$ there exists $u \in \lang_n(X)$ such that $wuv \in \lang(X)$.
\end{definition}

It is well known that the above definition of irreducibility agrees with the notion of irreducibility for topological dynamical systems.

\subsection{Shifts of Finite Type}\label{sec:sft}

Let $\mathcal{A}$ be alphabet as defined above, and let $F \subseteq \mathcal{L}(\mathcal{A}^\mathbb{Z})$ be a finite set of words in $\mathcal{A}$ which we call a set of \emph{forbidden words}.  Then let $X_F \subseteq \mathcal{A}^\mathbb{Z}$ be the set \[ X_F = \{ x \in \mathcal{A}^\mathbb{Z} \mid w \not\sqsubseteq x \text{ for all } w \in F \} \,. \]  Note that $X_F$ is a shift invariant subset of $\mathcal{A}^\Z$.  A shift space $X_F$ constructed in this way is called a \emph{shift of finite type (SFT)}, see \cite{lindmarcus} for more details and an introduction to SFTs.

A useful way to construct SFTs is as follows.  Let $G = (V,E)$ be a finite directed graph, then let $X_G$ be the space of all bi-infinite paths on $G$ which we topologize as subspace of $E^\mathbb{Z}$ with  the product topology.  Let us also denote as $r,s: E \to V$ the functions which send each edge on $G$ to its range and source vertex respectively.  Then $X_G$, which we call the \emph{edge shift} of $G$, is a shift of finite type since we must only forbid words of the form $F = \{ef \mid e,f \in E \text{ and } r(e) \neq s(f) \}$, which is a finite subset of $\mathcal{L}(E^\mathbb{Z})$.

Besides this being a useful construction, it is actually the case that any SFT is conjugate to the edge shift of some graph $G$.  Let $X_F$ be a SFT, then since $F$ is a finite set there is an $N \geq 1$ such that $|w| \leq N$ for all $w \in F$.  We may then construct a graph $G = (V,E)$ by letting $V = \mathcal{L}_{N-1}(X_F)$, $E = \mathcal{L}_N(X_F)$, $r(e) = e_2 e_3 \cdots e_N$, and $s(e) = e_1 e_2 \cdots e_{N-1}$.  Such a graph $G$ is called a \emph{presentation} of $X_F$ and it is not necessarily unique.  By a result in \cite[Section 2.3]{lindmarcus}, $X_F$ is conjugate to $X_G$.

\begin{example}\label{ex:full-2-shift} (\cite[Definition 1.1.1]{lindmarcus}) 
For example, every \emph{full $n$-shift}, that is the shift space $X_n = \{\word{0}, \word{1}, \dots, \word{n-1} \}^\mathbb{Z}$, is a shift of finite type.  The full $n$-shift can be built from a graph with a single vertex and $n$ edges.
\end{example}

\begin{example}\label{ex:golden-mean-shift} (\cite[Example 1.2.3]{lindmarcus})
Consider the shift space $X \subseteq \{\word{a}, \word{e}, \word{f} \}^\mathbb{Z}$ given by the following graph $G$.
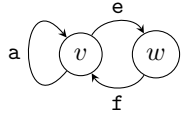
\begin{center}
\begin{tikzpicture}
\tikzset{every loop/.style={looseness=6, in=130, out=230}}
\pgfmathsetmacro{\S}{0.9}

\node[shape=circle, draw=black] (A) at (0,0) {$v$};
\node[shape=circle, draw=black] (B) at (1,0) {$w$};
\path[->,>=stealth] (A) edge[loop left] node[scale=\S, left] {$\word{a}$} (A);
\path[->,>=stealth] (A) edge[out=60, in=120] node[scale=\S, above] {$\word{e}$} (B);
\path[->,>=stealth] (B) edge[out=240, in=300] node[scale=\S, below] {$\word{f}$} (A);
\end{tikzpicture}    
\end{center}
This is a shift space called the \emph{golden mean shift}.  It is conjugate to the shift of finite type $X_F \subseteq \{\word{0}, \word{1}\}^\mathbb{Z}$ with the set of forbidden words $F = \{ \word{11} \}$.  To see this observe that if we construct the edge shift for $X_F$ with $N = 2$, we obtain the above graph with the labeling $\word{a} = \word{00}$, $\word{e} = \word{01}$, and $\word{f} = \word{10}$ --- see \cite[Section 2.3]{lindmarcus} for more details.
\end{example}

\subsection{Sofic Shifts}\label{sofic-shifts}

Similar to shifts of finite type, sofic shifts are constructed from a graph presentation except now we will consider graphs with labeled edges.  Let $G = (V,E)$ be a finite directed graph as before and let $\pi : E \to \mathcal{A}$ be a function which assigns to each edge $e \in E$ a label $\pi(e) \in \mathcal{A}$.  The function $\pi$ extends to a continuous function $\pi : X_G \to \mathcal{A}^\mathbb{Z}$ where $X_G$ is the SFT constructed from the unlabeled graph $G$.  Let $X$ denote the image of $X_G$ under $\pi$, then $\pi:X_G \to X$ is a factor map from $(X_G,\sigma)$ to $(X,\sigma)$.  The function $\pi$ also extends to a function $\pi:\lang_n(X_G) \to \lang_n(X)$ which assigns to each finite path in $G$ the word \[ \pi(e_1 e_2 \cdots e_n) = \pi(e_1) \pi(e_2) \cdots \pi(e_n) \] in $\lang(X)$.  We say that a path $p \in \lang_n(X_G)$ \emph{represents} a word $w \in X$ if $\pi(p) = w$.  Shift spaces constructed in this way are called \emph{sofic shifts}.  Note that every shift of finite type is also a sofic shift.  Sofic shifts can be succinctly characterized as those shift spaces that are factors of SFTs \cite[Section 3.1]{lindmarcus}.

\begin{example}\label{even-shift} (\cite[Example 1.2.4]{lindmarcus})
One of the running examples in this paper will be the \emph{even shift}, which is the sofic shift space in the alphabet $\{\word{0}, \word{1}\}$ whose elements do not contain any of the finite words in the set $F = \{\word{10}^{2k+1}\word{1} \mid k \geq 0\}$.  That is to say that the elements of the even shift are infinite binary sequences which have an even number of consecutive zeros between any two ones.  We will denote the even shift as $X_\ev$.  We can see that the even shift is not a shift finite type since there is no upper bound on the size of the words that we must forbid.  More precisely, the countable set $F$ is a minimal set of forbidden words for $X_\ev$.  However, the even shift is indeed sofic since it admits the following graph presentation.

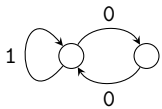
\vspace{1em}
\begin{center}
\begin{tikzpicture}
\tikzset{every loop/.style={looseness=10, in=130, out=230}}
\pgfmathsetmacro{\S}{0.9}

\node[shape=circle, draw=black] (A) at (0,0) {};
\node[shape=circle, draw=black] (B) at (1,0) {};
\path[->,>=stealth] (A) edge[loop left] node[scale=\S, left] {$\word{1}$} (A);
\path[->,>=stealth] (A) edge[out=60, in=120] node[scale=\S, above] {$\word{0}$} (B);
\path[->,>=stealth] (B) edge[out=240, in=300] node[scale=\S, below] {$\word{0}$} (A);
\end{tikzpicture}    
\end{center}
In this case, the edge shift $X_G$ coming the unlabeled graph above is exactly the golden mean shift discussed in Example \ref{ex:golden-mean-shift}.  In other words, the even shift is a factor of the golden mean shift.

It is important to note that the factor map $\pi : X_G \to X$ may not be injective.  For example in the even shift there are two elements of $X_G$ that get mapped to the point $\overline{\word{0}} = \dots \word{00.000} \dots$.  However in this case this is the only such point since as soon there is a $\word{1}$ in a sequence $x \in X$, we know exactly which element in $X_G$ must correspond to it.  This follows from the fact that there is only one vertex with an arrow labeled $\word{1}$ pointing to it.  Hence $\pi$ is injective on almost all points in $X_G$ except the two points that get mapped to $\overline{\word{0}}$.
\end{example}

\begin{definition}\label{def:resolving-map}
Suppose $(G, \pi)$ is a labelled graph and $X$ is the associated to sofic shift. We call $(G, \pi)$ right-resolving if for each vertex $I\in V$, the restriction of $\pi$ to the set $\{ e\in E \mid s(e)=I \}$ is injective. Such a labelled graph is called a right-resolving presentation of $X$. We call a right-resolving presentation minimal if $G$ has the fewest vertices among all right-resolving presentations of the sofic shift. We also call the shift of finite type associated to $G$, $\Sigma_G$, along with the map induced from $\pi$, the minimal right-resolving cover of $X$.

Likewise, we call $(G, \pi)$ left-resolving if for each vertex $I\in V$, the restriction of $\pi$ to the set $\{ e\in E \mid r(e)=I \}$ is injective and call such a labelled graph a left-resolving presentation. The definition of minimal and minimal cover are as in the previous paragraph with left replacing right.
\end{definition}

\begin{example}\label{even-shift-left-right}
One can check that the presentation of the even shift given in Example \ref{even-shift} is both left and right resolving. 
\end{example}

\begin{definition}\label{def:almost-finite-type}
Sofic shifts that admit a presentation that is both left and right resolving are called of almost finite type. 
\end{definition}

The next theorem is due to Fischer, see \cite[Section 3.3]{lindmarcus}.
\begin{theorem}
Suppose that $X$ is an irreducible sofic shift. Then $X$ admits a unique (up to natural isomorphism) minimal right-resolving presentation. Moreover, a similar statement holds for left-resolving and if $X$ is almost finite type, then it admits a unique minimal presentation that is both right and left resolving.
\end{theorem}

\section{Synchronizing Shifts}\label{sec:sync-words}

As above, $(X,\sigma)$ denotes a shift space.

\begin{definition}\label{def:synchronizing-shift}
Let $X$ be a shift space and $w \in \lang(X)$, then $w$ is called \emph{synchronizing} if, for any two words $u,v \in \lang(X)$, $uw \in \lang(X)$ and $wv \in \lang(X)$ then $uwv \in \lang(X)$.  An element $x \in X$ is called \emph{synchronizing} if it contains a synchronizing word.  A shift space $X$ is called \emph{synchronizing} if it is irreducible and contains at least one synchronizing element (or, equivalently, if $\lang(X)$ contains a synchronizing word).
\end{definition}

The concept of synchronizing for shift spaces can be found in the symbolic dynamics literature in various forms.  In \cite{deeley2022} we have generalized this definition for expansive topological dynamical systems.  The above definition in particular can be found in \cite{fiebig91, blanchard86}.  In the literature synchronizing words have also been called ``intrinsically synchronizing'' \cite{lindmarcus}, ``magic'' \cite{samuel98}, or ``finitary'' \cite{krieger84}.

\begin{example}\label{ex:even-shift-non-sync-pts}
For the even shift (Example \ref{even-shift}), a word is synchronizing if and only if it contains a $\word{1}$.  This is because the context of a word $w$ in the even shift only depends on the parity of the number of zeros on the left or right side of $w$, except in the special case when $w = \word{0}^k$.  In fact, the point of all zeros, $\overline{\word{0}} = \ldots \word{00.000} \ldots$, is the only non-synchronizing point in the even shift.  In particular, the even shift is a synchronizing shift space.  To show it is irreducible, it is easy to show that arbitrary $w,v \in \lang(X)$ can be glued together with $u = \word{0}^k$ ensuring the parity of $k$ produces a valid word in the even shift.
\end{example}

We remark that all irreducible sofic shifts are indeed synchronizing, see \cite[Corollary 3.11]{deeley2022}.

We will now describe the connection between synchronizing element of a shift space in the sense of Definition \ref{def:synchronizing-shift} and the definition of synchronizing given in \cite[Definition 3.2]{deeley2022}.

\begin{definition}\label{def:context}
Let $X$ be a shift space and $w \in \lang(X)$.  The set \[E(w) = \{(a,b) \in \lang(X) \times \lang(X) \mid awb \in \lang(X) \} \] is called the \emph{context} of $w$.  Note that for any $w,v \in \lang(X)$ with $E(w) = E(v)$ and $(a,b) \in E(w)$, we have $E(awb) = E(avb)$.
\end{definition}

Alternatively, we can see that $w$ is synchronizing if it is the case that \[ E\left(w\right) = A \times B \] where $A, B \subseteq \mathcal{L}(X)$ --- this is simply a reformulation of the definition.  Note also that if $w$ is synchronizing then any word containing $w$ is also synchronizing.

Consider $X$ the full $n$-shift.  Define the \emph{local stable and unstable} sets as (respectively)
\begin{align*}
    X^\text{s}(x, N) &= \{ y \in X \mid y_n = x_n \,\forall n \geq N \} \text{ , and } \\
    X^\text{u}(x, N) &= \{ y \in X \mid y_n = x_n \,\forall n \leq -N \} \,.
\end{align*}
The bracket map in \cite[Definition 2.3]{deeley2022}, $[-,-] : X^\text{u}(x, N) \times X^\text{s}(x, N) \to C\left(x_{[-N,N]}\right)$, is then given by
\[ ([y,z])_n = \begin{cases} y_n & n > N \\ x_n = y_n = z_n & |n| \leq N \\ z_n & n < -N \end{cases} \,. \]
The bracket map is always defined when $X$ is the full $n$-shift.  However if we restrict the bracket map to an arbitrary shift space $X$, the bracket map may not be defined in general.  We now state \cite[Definition 3.2]{deeley2022} in the case of shift spaces.
\begin{definition}\label{def:sync-top}
Suppose $X$ is an irreducible shift space.  Then $x \in X$ is called \emph{synchronizing} if there exists $N \geq 0$ such that $[-,-]$ is defined and is a homeomorphism.
\end{definition}

\begin{lemma}\label{lem:same-def}
For shift spaces, Definition \ref{def:sync-top} agrees with Definition \ref{def:synchronizing-shift}.
\end{lemma}
\begin{proof}
If a point $x \in X$ is synchronizing in the sense of Definition \ref{def:sync-top}, then there is a cylinder set $C(x_{[-N,N]})$ such that the bracket map is defined as above.  It then follows that $E(x_{[-N,N]})$ is a Cartesian product and hence the word $x_{[-N,N]}$ is synchronizing.

If $X$ contains a synchronizing point in the sense of Definition \ref{def:synchronizing-shift}, then there is an $N$ such that $x_{[-N,N]}$ is a synchronizing word.  This implies that $E(x_{[-N,N]})$ is a Cartesian product and hence the bracket map $[-,-]$ is defined and is indeed a homeomorphism.
\end{proof}

Hence synchronizing shifts are synchronizing dynamical systems in the sense of \cite{deeley2022}.  We will write $X$ as a disjoint union of its synchronizing and non-synchronizing points:
\[ X = X_\text{sync} \sqcup X_\text{non-sync} \]
Note that if $X$ is a synchronizing shift then the $X_\text{non-sync}$ is closed and shift invariant \cite[Lemma 3.5]{deeley2022}.  Hence $X_\text{non-sync}$ is a shift space which is determined by the set of non-synchronizing words (on this last point, see \cite{lindmarcus}).

\section{Local Conjugacy}\label{sec:local-conjgacy}

Local conjugacy is an equivalence relation which relates the local asymptotic behavior of elements in a dynamical system.  Thomsen defined local conjugacy for expansive dynamical systems in \cite{thomsen2010c}.  In \cite{deeley2022} we show that the property of being synchronizing is invariant under local conjugacy \cite[Proposition 4.6]{deeley2022}.  In the following, we will be using the definitions given in \cite[Section 4]{deeley2022}.  We review the definitions here, which are mainly due to Thomsen.  These equivalence relations are stated in terms of the topology on a shift space, but we will show that there are equivalent characterizations in terms of the language of a shift space.

\begin{definition}
Let $X$ be a shift space. We say that two points $x,y \in X$ are \emph{locally conjugate}, denoted $x \sim_\lc y$, if there exist two open neighborhoods $U$ and $V$ of $x$ and $y$ respectively, and a homeomorphism $\gamma : U \to V$ such that $\gamma(x) = y$ and \[ \lim_{n \to \pm \infty} \sup_{z \in U} \, d(\sigma^n(z), \sigma^n(\gamma(z))) = 0 \,. \]
\end{definition}

\begin{definition}
Let $X$ be a shift space with $x,y \in X$.  Suppose $N,N' \geq 0$ and $\gamma : X^\text{u}(x,N) \to X^\text{u}(y,N')$ is a homeomorphism onto its image such that $\gamma(x) = y$ and \[ \lim_{n \to \infty} \sup_{z \in X^\text{u}(x,N)} d(\sigma^n(z), \sigma^n(\gamma(z))) = 0 \,. \]  Then $\gamma$ is a called a \emph{stable local conjugacy} and we say that $x \sim_\lcs y$.
\end{definition}

\begin{definition}
Let $X$ be a shift space with $x,y \in X$.  Suppose $N,N' \geq 0$ and $\gamma : X^\text{s}(x,N) \to X^\text{s}(y,N')$ is a homeomorphism onto its image such that $\gamma(x) = y$ and \[ \lim_{n \to \infty} \sup_{z \in X^\text{s}(x,N)} d(\sigma^{-n}(z), \sigma^{-n}(\gamma(z))) = 0 \,. \]  Then $\gamma$ is a called an \emph{unstable local conjugacy} and we say that $x \sim_\lcu y$.
\end{definition}

Let $X$ be a shift space and $x,y \in X$, then we will also refer to the following equivalence relations.
\begin{itemize}
    \item $x \sim_\text{s} y$ if and only if there exists $N \in \mathbb{Z}$ such that $x_n = y_n$ for all $n \geq N$.
    \item $x \sim_\text{u} y$ if and only if there exists $N \in \mathbb{Z}$ such that $x_n = y_n$ for all $n \leq N$.
    \item $x \sim_\text{h} y$ if and only if there exists $N \geq 0$ such that $x_n = y_n$ for all $n \geq N$ and $n \leq -N$.
\end{itemize}
Note that clearly $x \sim_\lc y$ implies $x \sim_\text{h} y$ for any expansive dynamical system.  However the converse is not true, see Example \ref{ex:h-does-not-imply-lc}.  The following result from \cite{thomsen2010c} gives necessary and sufficient conditions for $x \sim_\lc y$ in a shift space.

\begin{lemma}\label{lem:shift-lc-conditions}
If $(X,\sigma)$ is a shift space, then two points $x,y \in X$ are locally conjugate if and only if there is an $N$ such that $x_n = y_n$ for all $|n| > N$ and $E\left(x_{[-N,N]}\right) = E\left(y_{[-N,N]}\right)$.  If the latter is true, we can construct a local conjugacy $\left(C\left(x_{[-N,N]}\right), C\left(y_{[-N,N]}\right), \gamma\right)$ from $x$ to $y$ where $\gamma$ is defined by \[ \gamma(z)_n = \begin{cases} z_n & |n| > N \\ y_n & |z| \leq N \end{cases} \,. \]
\end{lemma}

\begin{example}\label{ex:h-does-not-imply-lc}
Here we give a counterexample, due to Thomsen \cite[Remark 1.13]{thomsen2010c}, to the statement $x \sim_\text{h} y$ implies $x \sim_\lc y$.  Consider the even shift as in Example \ref{even-shift}.  Note that since in a shift space $x \sim_\text{h} y$ implies $x_n = y_n$ for all $|n| \geq N$ for some $N$, the sequences $x = \ldots \word{00.000} \ldots$ and $y = \ldots \word{00.100} \ldots$ are homoclinic in the even shift.  However, observe that $(\word{1},\word{0})$ is in $E\left(\word{0}^{4k-1}\right)$ but not in $E\left(\word{0}^{2k-1}\word{10}^{2k-1}\right)$ for all $k \geq 1$, so although $x$ is homoclinic to $y$ they cannot be locally conjugate.  This can also be readily seen from the fact that $y$ is synchronizing but $x$ is not, and any point locally conjugate to $y$ must also be synchronizing.
\end{example}

\subsection{Stable/Unstable Local Conjugacy}\label{sec:shift-lcs-lcu}

We must first introduce a definition similar to Definition \ref{def:context} except we now specialize to either the left or the right side of a word.

\begin{definition}
Let $X$ be a shift space and $w \in \lang(X)$.  Then we define the sets
\begin{align*}
    E^-(w) &= \{ a \in \lang(X) \mid aw \in \lang(X) \} \text{ , and } \\
    E^+(w) &= \{ b \in \lang(X) \mid wb \in \lang(X) \}
\end{align*}
called the \emph{left context} and \emph{right context} of $w$, respectively.
\end{definition}

There is an important caveat to the above definition that we must grapple with.  It is not necessarily the case for every synchronizing shift space that the set of words that can be appended to the left-infinite ray $x_{(-\infty,0]}$, \ie the \emph{follower set} of $x_{(-\infty,0]}$, is the same as $x_{[-N, 0]}$ for any $N \geq 0$.  For example, in the $\word{a}^n\word{b}^n$-shift (see Section \ref{context-free-shift}), one cannot append the word $\word{b}^k\word{a}$ to the right of the left-infinity ray $...\word{aaa}$ for any $k \geq 1$, and yet the words $\word{a}^k$ and $\word{b}^k\word{a}$ can always be concatenated.

However, consider the case where $X$ is a synchronizing shift space and $x \in X_\sync$ with $N$ such that $x_{[-N,N]}$ is a synchronizing word.  Then it is the case that follower set of $x_{(-\infty,N]}$ is the same as $E^+\left(x_{[-N,N]}\right)$.  To see this observe that for any $b \in E^+\left(x_{[-N,N]}\right)$, we have $x_{[-M, -N-1]} x_{[-N,N]} b$ is a valid word in $X$ for any $M > N$ since $x_{[-N,N]}$ is synchronizing.

This motivates the following lemma.  We remind the reader that in Lemma \ref{lem:shift-lcs-lcu-conditions} below we are referring to the stable and unstable local conjugacy relations as defined in \cite[Section 4]{deeley2022}.

\begin{lemma}\label{lem:shift-lcs-lcu-conditions}
Suppose $(X,\sigma)$ is a synchronizing shift space with $x,y \in X_\sync$.  Then the following are true.
\begin{enumerate}[(i)]
    \item $x \sim_\lcs y$ if and only if there exists $N \geq 0$ such that both $x_{[-N,N]}$ and $y_{[-N,N]}$ are synchronizing words, $x_n = y_n$ for all $n > N$, and $E^+\left(x_{[-N,N]}\right) = E^+\left(y_{[-N,N]}\right)$.
    \item $x \sim_\lcu y$ if and only if there exists $N \geq 0$ such that both $x_{[-N,N]}$ and $y_{[-N,N]}$ are synchronizing words, $x_n = y_n$ for all $n < -N$, and $E^-\left(x_{[-N,N]}\right) = E^-\left(y_{[-N,N]}\right)$.
\end{enumerate}
If the latter is true in (i), we can construct an explicit stable local conjugacy $\gamma : X^\text{u}(x, N) \to X^\text{u}(y, N)$ by \[ \gamma(z)_n = \begin{cases} z_n & n > N \\ y_n & n \leq N \end{cases} \] for $z \in X^\text{u}(x, N)$.  Similarly, if the latter is true in (ii), we can construct an explicity unstable local conjugacy $\gamma : X^\text{s}(x, N) \to X^\text{s}(y, N)$ by \[ \gamma(z)_n = \begin{cases} y_n & n \geq -N \\ z_n & n < -N \end{cases} \] for $z \in X^\text{s}(x, N)$.  
\end{lemma}
\begin{proof}
We will prove (i).  It is easy to see that the map $\gamma : X^\text{u}(x, N) \to X^\text{u}(y, N)$ is indeed a stable local conjugacy.  Conversely, if $x \sim_\lcs y$ via a stable local conjugacy $\gamma$, there is an $N$ such that $X^\text{u}(x, N)$ is contained in the domain of $\gamma$, $x_{[-N,N]}$ and $y_{[-N,N]}$ are both synchronizing words, and $z_n = \gamma(z)_n$ for all $n > N$ and $z \in X^\text{u}(x, N)$.  In particular $x_n = y_n$ for all $n > N$.  Note that $\gamma$ is a homeomorphism from $X^\text{u}(x, N)$ to $X^\text{u}(y, N)$.

We now show that $\gamma$ induces a bijection $E^+\left(x_{[-N,N]}\right) = E^+\left(y_{[-N,N]}\right)$.  For any $b \in E^+\left(x_{[-N,N]}\right)$, with $|b| = M$, there is a $z \in X^\text{u}(x, N)$ such that $z_{[N+1, N+M+1)} = b$.  Observe that $\gamma(z)_{[N+1, N+M+1)} = b$ implying $b \in E^+\left(y_{[-N,N]}\right)$, so we have shown $E^+\left(x_{[-N,N]}\right) \subseteq E^+\left(y_{[-N,N]}\right)$.  Since $\gamma$ itself is in particular a bijection, we similarly have $E^+\left(y_{[-N,N]}\right) \subseteq E^+\left(x_{[-N,N]}\right)$, thus proving (i).
\end{proof}

\begin{example}\label{ex:s-does-not-imply-lcs}
Consider the even shift (see Example \ref{even-shift}), and consider the two elements $x = ... \word{111.0000} ...$ and $y = ... \word{111.1000} ...$.  Note that $x \sim_\text{s} y$ and also that both $x$ and $y$ are synchronizing.  However, since $E^+\left(\word{1}^k\word{0}^{k+1}\right) \neq E^+\left(\word{1}^{k+1}\word{0}^k\right)$ for all $k \geq 1$, we cannot have $x \sim_\lcs y$.  Hence $X^\text{s}(x) \neq X^\lcs(x)$ for this particular choice of $x$.
\end{example}

\begin{example}\label{ex:lcu-lcs-sync-existence-counterexample}
Let $X \subseteq \{\word{a}, \word{b}, \word{c}\}^\Z$ be the shift space defined as the closure of the set of bi-infinite paths on the graph $G$ below.  This is a modified version of the $\word{a}^n\word{b}^n$-shift to be discussed in Section \ref{context-free-shift}.
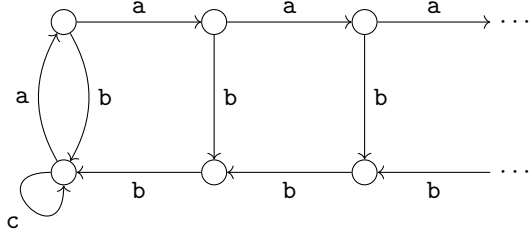
\vspace{1em}
\begin{center}
\begin{tikzpicture}[scale=2]

\tikzset{every loop/.style={looseness=0, in=270, out=180}}
\tikzset{vertex/.style = {shape=circle, draw}}
\tikzset{dots/.style = {}}
\tikzset{edge/.style = {->}}

\node[vertex] (v1) at (0, 1) {};
\node[vertex] (v2) at (1, 1) {};
\node[vertex] (v3) at (2, 1) {};
\node[dots] (v4) at (3, 1) {\(\cdots\)};
\node[vertex] (w1) at (0, 0) {};
\node[vertex] (w2) at (1, 0) {};
\node[vertex] (w3) at (2, 0) {};
\node[dots] (w4) at (3, 0) {\(\cdots\)};

\draw[edge] (v1) to[bend left] node[right]{$\word{b}$} (w1);
\draw[edge] (v1) to node[above]{$\word{a}$} (v2);
\draw[edge] (v2) to node[right]{$\word{b}$} (w2);
\draw[edge] (v2) to node[above]{$\word{a}$} (v3);
\draw[edge] (v3) to node[right]{$\word{b}$} (w3);
\draw[edge] (v3) to node[above]{$\word{a}$} (v4);
\draw[edge] (w1) to[bend left] node[left]{$\word{a}$} (v1);
\draw[edge] (w2) to node[below]{$\word{b}$} (w1);
\draw[edge] (w3) to node[below]{$\word{b}$} (w2);
\draw[edge] (w4) to node[below]{$\word{b}$} (w3);
\draw[->] (w1) edge[loop below] node[below left]{$\word{c}$} (w1);

\end{tikzpicture}
\end{center}

The shift space $X$ is a mixing synchronizing system.  The proof goes as follows.  This shift space is synchronizing since $\word{ba}$ is a synchronizing word, see Lemma \ref{lem:cf-shift-is-sync}.  We use the following definition of mixing for shift spaces: suppose $w,v \in \lang(X)$, then there is an $N$ such that for each $n \geq N$ there is $u$ with $|u| = n$ and $wuv \in \lang(X)$.  Let $v_0$ denote the bottom-left vertex in the graph of $G$ which has an edge labeled $\word{c}$ from $v_0$ to itself.  Consider the words $w,v \in \lang(X)$, and let $p$ and $q$ be finite paths on $G$ such that their labels are $w$ and $v$ respectively.  Let $v_1$ denote the target vertex of $p$ and $v_2$ the initial vertex of $q$.  Observe that the graph $G$ is irreducible, so we can find a path $p'$ from $v_1$ to $v_0$, and a path $q'$ from $v_0$ to $v_2$.  Let $w'$ be the label of $p'$ and $v'$ be the label of $q'$.  Then we can see that $ww'\word{c}^kv'v$ is a word in $\lang(X)$ for any $k \geq 0$, hence $X$ is mixing.

Let $p = \overline{\word{ab}} = \ldots \word{ab.aba} \ldots$, which is a synchronizing periodic point in $X$.  Consider the points
\begin{align*}
    x &= \ldots \word{ababaa.aaaaaa} \ldots \text{ , and } \\ 
    y &= \ldots \word{bbbbbb.bbabab} \ldots
\end{align*}
where $x \in X^\text{u}(p)$ and $y \in X^\text{s}(p)$.  Then in order for $z$ to satisfy $x \sim_\lcs z \sim_\lcu y$, by Lemma \ref{lem:shift-lcs-lcu-conditions} we need there to exist $N \geq 0$ such that
\begin{itemize}
    \item $x_{[-N,N]}$, $y_{[-N,N]}$, and $z_{[-N,N]}$ are synchronizing,
    \item $z_n = x_n$ for all $n > N$,
    \item $z_n = y_n$ for all $n < N$,
    \item $E^+\left(z_{[-N,N]}\right) = E^+\left(x_{[-N,N]}\right)$, and
    \item $E^-\left(z_{[-N,N]}\right) = E^-\left(y_{[-N,N]}\right)$.
\end{itemize}
For $N \geq 2$, $x_{[-N,N]} = \ldots \word{ba}^{N+3}$ and $y_{[-N,N]} = \word{b}^{N+2}\word{a} \ldots$ are synchronizing.  Observe that for any such $N$, we need $z_{[-N,N]}$ to be such that $E^+\left(z_{[-N,N]}\right) = E^+\left(x_{[-N,N]}\right)$ and $E^-\left(z_{[-N,N]}\right) = E^-\left(y_{[-N,N]}\right)$.  Furthermore, we can see from the graph $G$ that $E^+(w) = E^+\left(x_{[-N,N]}\right)$ if and only if $w = w'\word{ba}^{N+2}$ for some other word $w'$.  Likewise, $E^-(w) = E^-\left(y_{[-N,N]}\right)$ if and only if $w = \word{b}^{N+1}\word{a}w'$ for some other word $w'$.  Hence $z_{[-N,N]}$ must have the form \[ z_{[-N,N]} = \word{b}^{N+2}w\word{a}^{N+3} \,. \]  However, this is a contradiction since $|z_{[-N,N]}| = 2N + 1$ but the above conditions say that we need at least $|z_{[-N,N]}| \geq 2N + 5$.  Hence no such $z$ satisfying $x \sim_\lcs z \sim_\lcu y$ can exist.
\end{example}

\subsection{Stable/Unstable Local Conjugacy for Sofic Shifts}

The proof of the next theorem uses results from \cite[Section 3.3]{lindmarcus}. This is the only proof that uses these results, so we will not introduce them in detail. We note that in \cite{lindmarcus} the term ``synchronizing with respect to a presentation" is used. Exercises 3.3.3 and 3.3.4 in \cite{lindmarcus} imply that ``synchronizing with respect to the minimal right-resolving presentation" is the same as ``synchronizing" as used in the present paper.

\begin{theorem} \label{sofic-lemma-resolving}
Suppose that $X$ is an irreducible sofic shift, $(G, \pi)$ is its minimal right-resolving presentation, and $\Sigma_G$ is the shift of finite type associated to the graph $G$. If $a$, $b$ are in $\Sigma_G$ with $\pi(a)$ and $\pi(b)$ synchronizing and $\pi(a) \sim_\lcs \pi(b)$ in $X$, then $a \sim_s b$ in $\Sigma_G$. Moreover, a similar results holds in the case of the minimal left-resolving presentation and $\sim_\lcu$.
\end{theorem} 
\begin{proof}
Write $a$ and $b$ as $(a_n)_{n\in \Z}$ and $(b_n)_{n\in \Z}$ respectively. Also, let $x_n=\pi(a_n)$ and $y_n=\pi(b_n)$ where $\pi$ is the map that takes an edge to its label.

By Lemma \ref{lem:shift-lcs-lcu-conditions}, there exists $N \in \N$ such that 
\begin{enumerate}
\item both $x_{[-N,N]}$ and $y_{[-N,N]}$ are synchronizing words, 
\item $x_n = y_n$ for all $n > N$, and 
\item $E^+\left(x_{[-N,N]}\right) = E^+\left(y_{[-N,N]}\right)$.
\end{enumerate}

By \cite[Lemma 3.3.15]{lindmarcus}, $E^+\left(x_{[-N,N]}\right)=E^+(t(x_N))$ and $E^+\left(y_{[-N,N]}\right)=E^+(t(y_N))$. Hence, by item 3 above, $E^+(t(x_N))=E^+(t(y_N))$. By \cite[Proposition 3.3.9]{lindmarcus}, $(G, \pi)$ is follower-separated, which by definition implies that $t(x_N)=t(y_N)$. Hence $i(x_{N+1})=i(y_{N+1})$. Using item 2 above and the definition of right-resolving, it follows that $a_{N+1}=b_{N+1}$. 

An induction argument then implies that $a_n=b_n$ for all $n > N$ and hence that $a \sim_s b$.

\end{proof}

\section{The Homoclinic and Heteroclinic Algebras of Shift Spaces}\label{sec:shift-algebras}

In \cite{deeley2022} we defined several $C^\ast$ algebras associated to synchronizing systems, based on work by Thomsen \cite{thomsen2010c}.  These are built from the equivalence relations discussed in Section \ref{sec:local-conjgacy}, which may be topologized so as to produce \'{e}tale groupoids.

In the case of a shift space $X$, the homoclinic and heteroclinic algebras of $(X,\sigma)$ are AF-algebras and can be computed from certain Bratteli diagrams.  First we will discuss how one can construct an \'{e}tale groupoid from a Bratteli diagram.  Then we will show how to construct Bratteli diagrams which produce the homoclinic and heteroclinic algebras of a shift space.  Thomsen constructs the homoclinic algebra of a shift space in \cite{thomsen2010c}, and it is implicit that a similar construction is possible for the heteroclinic algebras of a shift space.

Bratteli diagrams were introduced by Brattli in \cite{Bratteli72}.  A Bratteli diagram $B = (V,E)$ is an indexed collection of vertices $V = \{V_n\}_{n \geq N}$ and an indexed collection of edges $E = \{E_n\}_{n \geq N}$.  For our purposes we will index the vertex and edge sets starting at some $N \in \mathbb{Z}$.  In a Bratteli diagram one thinks of the edges in $E_n \in E$ as going from vertices in $V_n$ to vertices in $V_{n+1}$.  For an edge $e \in E_n$, let $i(e) \in V_n$ denote the initial vertex of $e$ and $t(e) \in V_{n+1}$ denote the terminal vertex of $e$.

We will now construct a groupoid $G(B)$ from a Bratteli diagram $B$.  Let \[ B^\infty = \{ (e_N, e_{N+1}, \ldots) \mid t(e_n) = i(e_{n+1}) \text{ for all } n \geq N \} \] be the set of infinite paths in the Bratteli diagram $B$, topologized as subspace of $\displaystyle \prod_{n \geq N} E_n$.  Let $B(e, M) \subseteq B^\infty$ be defined as \[ B(e, M) = \{ f \in B^\infty \mid f_n = e_n \text{ for all } N \leq n \leq M \} \,. \]  The sets $B(e, M)$ are a clopen basis for the topology on $B^\infty$.  Define $G(B)$ to be the set \[ G(B) = \{ (e,e') \in B^\infty \times B^\infty \mid \exists M \geq N \text{ such that } e_n = e'_n \text{ for all } n > M \} \,. \] If $(e,e') \in G(B)$ and $e_n = e'_n$ for all $n > M$ for some $M$, then in particular $t(e_M)$ and $t(e'_M)$ are the same vertex.  Hence for each $(e,e') \in G(B)$ we have, for some $M$, a bijection $\gamma : B(e,M) \to B(e',M)$ defined by
\[ \gamma(f) = (e'_0, e'_1, \ldots, e'_M, f_{M+1}, \ldots) \,. \]
The sets $\Gamma(e, e', M) = \{ (f, \gamma(f)) \mid f \in B(e,M) \}$ are a basis for a topology on $G(B)$.  In fact, with this topology $G(B)$ is an \'{e}tale groupoid --- see \cite[Theorem 3.5.6]{putnamNotes19}.  For $(e,e') \in G(B)$ we define $s(e,e') = e$ and $r(e,e') = e'$, and remark that the unit space $G(B)^0$ is identified with $B^\infty$.

Given a Bratteli diagram $B = (V,E)$, one can produce an AF-algebra $C^\ast(B)$.  Conveniently, we may compute the $K$-theory of this AF-algebra from the Bratteli diagram itself, that is, \[ K_0\big(C^\ast(B)\big) \cong \varinjlim  \left(\Z V_n, A_n \right) \]  where $A_n : Z_n \to Z_{n+1}$ is the $|V_{n+1}|$ by $|V_n|$ matrix with non-negative integer entries encoding the edges between $V_n$ and $V_{n+1}$.  That is \[ A_n(v) = \sum_{e \in i^{-1}(v)} t(e) \] for all $v \in V_n$.  We call $A_n$ the \emph{transition matrices}.

\subsection{Homoclinic Algebra}\label{sec:shift-bratteli-homoclinic}

We denote by $\lang_n(X)$ the words of length $n$ in a shift space $X$, in other words $\lang_n(X) = \{ w \in \lang(X) \mid |w| = n \}$.  Recall that the context of $w \in \lang(X)$, denoted $E(w)$, is given in Definition \ref{def:context} as the set \[ E(w) = \{ (a,b) \in \lang(X) \times \lang(X) \mid awb \in \lang(X) \} \,. \]

\begin{definition}\label{def:shift-context-equivalence}
Let $X$ be a shift space.  Then for each $n \geq 0$ we define an equivalence relation on $\lang_n(X)$ by $w \sim v$ if and only if $E(w) = E(v)$.  Then we say $w$ and $v$ are \emph{context equivalent}.
\end{definition}

We will now construct a Bratteli diagram $B^\lc(X)$, where $X$ is a shift space over the alphabet $\mathcal{A}$.  Let $V_{-1} = \{ * \}$ be a singleton set, and let $V_n = \lang_{2n + 1}(X) /\sim$ where $\sim$ denotes the context equivalence relation as in Definition \ref{def:shift-context-equivalence}.  We define the edge set for $n = -1$ to be $E_{-1} = \{ (*, [a]) \mid a \in \mathcal{A} \}$.  Then for $n \geq 0$ we define the edge sets \[ E_n = \left\{ \big([w], [awb]\big) \,\middle|\, [w] \in V_n,\, a,b \in \mathcal{A} \text{ such that } (a,b) \in E(w) \right\} \,. \]  If $[v] = [w]$ for $w,v \in \lang_{2n+1}(X)$, then $E(w) = E(v)$ by definition.  If $a,b \in \mathcal{A}$ such that $(a,b) \in E(w)$, then $awb, avb \in \lang_{2n+3}(X)$ and furthermore $E(awb) = E(avb)$.  Hence the operation $[w] \mapsto [awb]$ is a well-defined operation on context equivalence classes whenever $(a,b) \in E(w)$.  In short, for each $a,b \in \mathcal{A}$ such that $(a,b) \in E(w)$, we get an edge from $[w] \in V_n$ to $[awb] \in V_{n+1}$ which we label as $(a,b)$.

\begin{theorem}\label{theorem:shift-bratteli-lc-isomorphism}
Suppose $X$ be a shift space.  Then $G\left(B^\lc(X)\right) \cong G^\lc(X,\sigma)$.
\end{theorem}
\begin{proof}
Let $x \in X$.  Then it is straightforward to see from the construction of $B^\lc(X)$ that $x$ corresponds uniquely to the path $\rho(x) = \{ e_n \}_{n \geq -1}$ where $t(e_{-1}) = [x_0]$ and $t(e_n) = [x_{[-n,n]}]$ for all $n \geq 0$.  In particular each $e_n$ is the edge $\big([x_{[-n,n]}], [x_{[-n-1,n+1]}]\big) \in E_n$.  Furthermore, any element $(x_0, (x_{-1}, x_1), (x_{-2}, x_2), \ldots) \in \left(B^\lc(X)\right)^\infty$ corresponds to the sequence $\ldots x_{-2} x_{-1} x_0 x_1 x_2 \ldots \in X$.  We denote this bijection $\rho : X \to \left(B^\lc(X)\right)^\infty$.

From Lemma \ref{lem:shift-lc-conditions}, we know that $x,y \in X$ satisfy $x \sim_\lc y$ if and only if there exists $M \geq 0$ such that $x_n = y_n$ for all $|n| > M$ and $E\left(x_{[-M,M]}\right) = E\left(y_{[-M,M]}\right)$.  This implies that $(\rho(x), \rho(y)) \in G\left(B^\lc(X)\right)$.  Conversely, if $(e,e') \in G\left(B^\lc(X)\right)$, then there is an $M$ such that $e_n = e'_n$ for all $n > M$.  Since $\rho$ is a bijection, $e = \rho(x)$ and $e' = \rho(y)$ for some $x,y \in X$.  Then since $t(e_M) = [x_{[-M,M]}] = [y_{[-M,M]}] = t(e'_M)$, we have $E\left(x_{[-M,M]}\right) = E\left(y_{[-M,M]}\right)$.  Lastly since $e_n = e'_n$ for all $n > M$, it follows that $x_n = y_n$ for all $n > M$, implying $x \sim_\lc y$.  Hence we have a bijection $\Phi : G^\lc(X,\sigma) \to G\left(B^\lc(X)\right)$ defined by $\Phi(x,y) = (\rho(x), \rho(y))$ which clearly preserves the groupoid structure.

Let $\gamma : C\left(x_{[-M,M]}\right) \to C\left(y_{[-M,M]}\right)$ denote a local conjugacy from $x$ to $y$ in $X$, then the set $\left\{ (z, \gamma(z)) \,\middle|\, z \in C\left(x_{[-M,M]}\right) \right\}$ is an open neighborhood of $(x,y)$ in $G^\lc(X,\sigma)$.  In fact, sets of this form are a basis for the topology on $G^\lc(X,\sigma)$ in the case when $X$ is a shift space.  Finally, to see that $\Phi$ is homeomorphism, see that \[ \Phi\left(\left\{ (z, \gamma(z)) \,\middle|\, z \in C\left(x_{[-M,M]}\right) \right\}\right) = \Gamma(\rho(x), \rho(y), M) \,. \] Hence $G\left(B^\lc(X)\right) \cong G^\lc(X,\sigma)$.
\end{proof}

From Definition \ref{def:synchronizing-shift} we can see that if $w \in \lang_{2n+1}(X)$ is synchronizing then the entire equivalence class $[w] \in V_n$ is synchronizing.  Hence we will call $[w] \in V_n$ a \emph{synchronizing vertex}.  Suppose $X$ is a synchronizing shift space and $x \in X_\sync$ is a synchronizing point.  Then for some $n$ the word $x_{[-n,n]}$ is synchronizing, so $[x_{[-n,n]}] \in V_n$ is a synchronizing vertex.  Furthermore if $w$ is synchronizing and $(a,b) \in E(w)$, then $awb$ is also synchronizing.  Thus if a path $e \in \left(B^\lc(X)\right)^\infty$ passes through a synchronizing vertex, then every subsequent vertex in $e$ must also be synchronizing.  Let $\rho : X \to \left(B^\lc(X)\right)^\infty$ be the bijection defined in the proof of Theorem \ref{theorem:shift-bratteli-lc-isomorphism}.  Then the image of $X_\sync$ under $\rho$ is the subset of all paths in $\left(B^\lc(X)\right)^\infty$ which eventually pass through a synchronizing vertex.  Hence, following Theorem \ref{theorem:shift-bratteli-lc-isomorphism}, $G^\lc_\sync(X,\sigma)$ (defined in Theorem \cite[Theorem 5.8]{deeley2022}) is isomorphic to the subgroupoid of all $(e,e') \in G\left(B^\lc(X)\right)$ such that $e,e' \in \rho\left(X_\sync\right)$.

\subsection{Heteroclinic Algebras}\label{sec:shift-bratteli-heteroclinic}

We will now construct a series of Bratteli diagrams which together will give us the groupoid $G^\lcs(X,\sigma,p)$ for a synchronizing shift space.  Let $(X,\sigma)$ be a synchronizing shift space over the alphabet $\mathcal{A}$, and let $p \in X$ be a synchronizing periodic point.  Since $p$ is periodic and synchronizing, then for each $N$ there exists $k = k(N) \geq N$ such that $p_{[-k,-N]}$ is a synchronizing word.  Hence the follower set of $p_{(-\infty,-N]}$ agrees with $p_{[-k,-N]}$, see Section \ref{sec:shift-lcs-lcu}.  We will denote $E^+\parens{p_{[-k,-N]}w}$ as $E^+\parens{p_{(-\infty,-N]}w}$ for any $w \in E^+\parens{p_{[-k,-N]}}$.

We now fix $N \geq 0$ and construct a Bratteli diagram $B^\lcs_N(X,p)$ which will produce all stable local conjugacy equivalence classes of elements in $X^\text{u}(p, N)$.  Recall $X^\text{u}(p,N)$ is defined in Section \ref{sec:shift-lcs-lcu} as the set \[ X^\text{u}(p,N) = \{ x \in X \mid x_n = p_n \text{ for all } n \leq -N \} \,. \]
For simplifying notation moving forward, denote $E^+_n\left(u\right) = \lang_n(X) \cap E^+\left(u\right)$ for $u \in \lang(X)$.  Let $V_{-N} = \{*\}$ be a singleton set.  Then for $n > -N$ let the set of vertices $V_n$ in the Bratteli diagram be \[ V_n = E^+_{n-N}\left(p_{(-\infty,-N]}\right) \big/ \sim \] where $w \sim v$ if and only if $E^+\left(p_{(-\infty,-N]}w\right) = E^+\left(p_{(-\infty,-N]}v\right)$.  We then define the set of edges $E_n$ from $V_n$ to $V_{n+1}$ as \[ E_n = \left\{ ([w], [wa]) \,\middle|\, [w] \in V_n, a \in E^+_1\left(p_{(-\infty,-N]}w\right) \right\} \,. \]
Similar to the homoclinic case, the operation $[w] \mapsto [wa]$ is a well-defined operation on equivalence classes.  In summary, for each $a \in \mathcal{A}$ such that $a \in E^+(w)$, we get an edge from $[w] \in V_n$ to $[wa] \in V_{n+1}$ which we think of as being labeled `a'.

Next note that for each $N$ we have an inclusion $\left(B^\lcs_N(X,p)\right)^\infty \subseteq \left(B^\lcs_{N+1}(X,p)\right)^\infty$ by the map \[ (e_{-N+1}, e_{-N+2}, \ldots) \mapsto (p_N, e_{-N+1}, e_{-N+2}, \ldots) \] where $p_N$ is the edge from the singleton set $V_{-(N+1)}$ to $[p_N] \in V_N$ in the Bratteli diagram $B^\lcs_{N+1}(X,p)$.  Hence, by our definition of the groupoid of a Bratteli diagram, we get an induced inclusion $G\left(B^\lcs_N(X,p)\right) \subseteq G\left(B^\lcs_{N+1}(X,p)\right)$ for each $N \geq 0$.

We omit the construction of $B^\lcu_N(X,p)$ since it is very similar (\ie it is built on the predecessor set of $p_{[N,\infty)}$ for each $N \geq 0$), but we remark that there are identical inclusions of $G\left(B^\lcu_N(X,p)\right)$ into $G\left(B^\lcu_{N+1}(X,p)\right)$.

\begin{theorem}\label{theorem:shift-bratteli-lcu-lcs-isomorphism}
Suppose $X$ is a synchronizing shift space.  Then
\begin{enumerate}[(i)]
    \item $\displaystyle G^\lcs(X,\sigma,p) \cong \bigcup_{N \geq 0} G\left(B^\lcs_N(X,p)\right)$, and
    \item $\displaystyle G^\lcu(X,\sigma,p) \cong \bigcup_{N \geq 0} G\left(B^\lcu_N(X,p)\right)$.
\end{enumerate}
\end{theorem}
\begin{proof}
We will show the proof of (i).  First we will show that for each $N \geq 0$ we have a bijection $\rho_N : X^\text{u}(p,N) \to \left(B^\lcs_N(X,p)\right)^\infty$.  For each $x \in X^\text{u}(p,N)$, by definition, $x_n = p_n$ for all $n \leq N$.  Hence it is easy to see from the construction of $B^\lcs_N(X,p)$ that the map $\rho_N$ defined by \[ \rho_N(x) = x_{(-N,\infty)} = (x_{-N+1}, x_{-N+2}, \ldots) \] is a bijection.  

Recall that $\displaystyle X^\text{u}(p) = \bigcup_{N \geq 0} X^\text{u}(p,N)$ by the definition of unstable equivalence on shift spaces.  We will show that $\rho_N$ is compatible with the inclusions $X^\text{u}(p, N) \subseteq X^\text{u}(p, N+1)$ and $\left(B^\lcs_N(X,p)\right)^\infty \subseteq \left(B^\lcs_{N+1}(X,p)\right)^\infty$, and hence the $\rho_N$ lift to a bijection $\displaystyle \rho : X^\text{u}(p) \to \bigcup_{N \geq 0} \left(B^\lcs_N(X,p)\right)^\infty$.  This compatibility is easy to see, as $\rho_{N+1}(x) = (p_{-N}, x_{-N+1}, x_{-N+2}, \ldots)$ for $x \in X^\text{u}(p, N) \subseteq X^\text{u}(p, N+1)$.  Similarly the inclusion of $\rho_N(x) \in \left(B^\lcs_N(X,p)\right)^\infty$ into $\left(B^\lcs_{N+1}(X,p)\right)^\infty$ is given by $\rho_N(x) = (x_{-N+1}, x_{-N+2}, \ldots) \mapsto (p_{-N}, x_{-N+1}, x_{-N+2}, \ldots)$.

If $x,y \in X^\text{u}(p)$ and $x \sim_\lcs y$, then by Lemma \ref{lem:shift-lcs-lcu-conditions} there exists $M \geq 0$ such that $x_{[-M,M]}$ and $y_{[-M,M]}$ are both synchronizing words, $x_n = y_n$ for all $n > M$, and $E^+\left(x_{[-M,M]}\right) = E^+\left(y_{[-M,M]}\right)$.  Furthermore, since $x,y \in X^\text{u}(p)$, there exists $N > M$ such that $x,y \in X^\text{u}(p,N)$.  Note that since $x_{[-M,M]}$ and $y_{[-M,M]}$ are both synchronizing words, we also have $E^+\left(p_{(-\infty,-N]}x_{(-N,M]}\right) = E^+\left(p_{(-\infty,-N]}y_{(-N,M]}\right)$.  In total, this means that the paths $\rho_N(x) = x_{(-N,\infty)}$ and $\rho_N(y) = y_{(-N,\infty)}$ satisfy $\rho_N(x)_n = \rho_N(y)_n$ for all $n > M$, hence $(\rho_N(x), \rho_N(y)) \in G\left(B^\lcs_N(X,p)\right)$.

Conversely, given $(e,e') \in G\left(B^\lcs_N(X,p)\right)$ for some $N \geq 0$, it is clear from the construction of $B^\lcs_N(X,p)$ that the points $\rho_N^{-1}(e)$ and $\rho_N^{-1}(e')$ are stably locally conjugate in $X^u(p,N)$.  Let $G^\lcs(X,\sigma,p)|_{X^\text{u}(p,N)}$ denote the set \[ G^\lcs(X,\sigma,p)|_{X^\text{u}(p,N)} = \{ (x,y) \in G^\lcs(X,\sigma,p) \mid x,y \in X^\text{u}(p, N) \} \,. \]  Then for each $N \geq 0$ we shown there is a bijection
\[ G^\lcs(X,\sigma,p)|_{X^\text{u}(p,N)} \mapsto G\left(B^\lcs_N(X,p)\right) \] via the map $(x,y) \mapsto (\rho_N(x), \rho_N(y))$.  This lifts to a bijection $\displaystyle \Phi : G^\lcs(X,\sigma,p) \to \bigcup_{N \geq 0} G\left(B^\lcs_N(X,p)\right)$ defined by $\Phi(x,y) = (\rho(x), \rho(y))$.  Similar to the proof of \ref{theorem:shift-bratteli-lc-isomorphism}, the map $\Phi$ is a homeomorphism, hence $\displaystyle G^\lcs(X,\sigma,p) \cong \bigcup_{N \geq 0} G\left(B^\lcs_N(X,p)\right)$.
\end{proof}

Lastly, we will show that in some cases it is sufficient (up to Morita equivalence) to compute $B^\lcs_0(X,p)$.  In particular it is the case with mixing sofic shifts, as shown in the following theorem.

\begin{theorem}\label{theorem:shift-bratteli-lcu-lcs-morita-equivalence}
Let $X$ be a mixing sofic shift.  Then,
\begin{enumerate}
    \item $G\left(B^\lcs_0(X,p)\right)$ and $G^\lcs(X,\sigma,p)$ are Morita equivalent, and
    \item $G\left(B^\lcu_0(X,p)\right)$ and $G^\lcu(X,\sigma,p)$ are Morita equivalent.
\end{enumerate}
\end{theorem}
\begin{proof}
We will show the proof in the stable case.  We use \cite[Definition 3.8]{farsi2018ample} and \cite[Theorem 3.12]{farsi2018ample}, and show that $G\left(B^\lcs_0(X,p)\right)$ and $G^\lcs(X,\sigma,p)$ are Kakutani equivalent, and hence Morita equivalent.  Observe from the proof of Theorem \ref{theorem:shift-bratteli-lcu-lcs-isomorphism} that the map $(x,y) \mapsto (\rho_0(x), \rho_0(y))$ gives an isomorphism $G^\lcs(X,\sigma,p)|_{X^\text{u}(p,0)} \to G\left(B^\lcs_0(X,p)\right)$, where $G^\lcs(X,\sigma,p)|_{X^\text{u}(p,0)}$ is defined as \[ G^\lcs(X,\sigma,p)|_{X^\text{u}(p,0)} = \{ \gamma \in G^\lcs(X,\sigma,p) \mid r(\gamma), s(\gamma) \in X^\text{u}(p, 0) \} \,. \]  Furthermore Note that $X^\text{u}(p,0)$ and $\left(B^\lcs_0(X,p)\right)^\infty$ are both clopen in the topology on the unit spaces of $G^\lcs(X,\sigma,p)$ and $G\left(B^\lcs_0(X,p)\right)$ respectively.  Lastly, we must now show that that $X^\text{u}(p,0) \subset X^\text{u}(p)$ is $G^\lcs(X,\sigma,p)$-full (see \cite{matui12}), that is to say for all $y \in X^\text{u}(p)$ there is some $x \in X^\text{u}(p, 0)$ such that $x \sim_\lcs y$.  This follows from Lemma \cite[Lemma 6.7]{deeley2022} since sofic shifts are finitely presented.
\end{proof}

\begin{definition}
In summary, given a synchronizing shift $X$ and a fixed synchronizing periodic point $p \in X$, we define the following $C^\ast$-algebras as in \cite{deeley2022}.

\begin{enumerate}[(i)]
    \item The \emph{homoclinic algebra} of $X$ is $A(X,\sigma) = C^\ast(G^\lc(X,\sigma))$.
    \item The \emph{stable synchronizing algebra} of $X$ is $S(X,\sigma,p) = C^\ast(G^\lcs(X,\sigma,p))$.
    \item The \emph{unstable synchronizing algebra} of $X$ is $U(X,\sigma,p) = C^\ast(G^\lcu(X,\sigma,p))$.
    \item The \emph{synchronizing ideal} of $X$ is \[ \I_\sync(X,\sigma) = C^\ast\left(G^\lc_\sync(X,\sigma)\right) = C^\ast\left(G^\lc(X,\sigma)|_{X_\sync}\right) \subseteq A(X,\sigma) \,. \]
\end{enumerate}

The stable and unstable synchronizing algebras are collectively called the \emph{heteroclinic synchronizing algebras} of $X$. We will sometime drop the shift map from the notation of these algebras. Also, it is worth noting that for shift spaces each of the relevant groupoids is AF and hence amenable, so the full and reduced $C^\ast$-algebras are natural isomorphic. It is not known if this result generalizes to all synchronizing systems. 
\end{definition}
Note that in general the heteroclinic synchronizing algebras are related to but not the same as, or even Morita equivalent to, the \emph{heteroclinic algebras} defined by Thomsen in \cite{thomsen2010c}.  This is due to the fact that Thomsen builds the heteroclinic algebras out of the set of all periodic points, while we are restricting to a single synchronizing periodic point.  By \cite[Theorem 5.5]{deeley2022}, the choice of synchronizing periodic point does not matter (up to Morita equivalence).

The next theorem relates the heteroclinic synchronizing algebras of an irreducible sofic shift with those of certain shifts of finite type. Since the $C^\ast$-algebras associated to a shift of finite type are well-understood, this allows us to get information about the algebras associated to a sofic shift. The proof uses results from \cite{deeley2022}.

\begin{theorem} \label{sofic-min-right-alg}
Suppose $X$ is a mixing sofic shift, $(G, \pi)$ is the minimal right-resolving presentation of $X$, and $\Sigma_G$ is the associated shift of finite type. Then there is exists synchronizing periodic point $p\in X$ such that $\pi^{-1}(p) = \{ q \}$ for some $q\in \Sigma_G$ and the map $\pi \times \pi : G^s(\Sigma_G, q) \rightarrow G^\lcs(X, p)$ is an isomorphism of groupoids. In particular, $C^\ast(G^s(\Sigma_G, q)) \cong  C^\ast(G^\lcs(X, p))$. A similar result holds for $G^\lcu(X, p)$ and the minimal left-resolving presentation of $X$.
\end{theorem}
\begin{proof}
For the first part, by \cite[Proposition 9.1.6]{lindmarcus}, the map $\pi : \Sigma_G \rightarrow X$ is almost one-to-one, so by \cite[Lemma 6.4]{deeley2022} there exists $p$ as required.

For the second part, by \cite[Lemma 6.8]{deeley2022}, $\pi \times \pi$ is a continuous open inclusion. Thus, we need only show it is onto. 

Let $(x, y) \in G^\lcs(X, p)$. By \cite[Lemma 6.5]{deeley2022}, the restriction to the unit space is a homeomorphism and so there exists $a$ and $b$ in $X^u(q)$ such that $\pi(a)=x$ and $\pi(b)=y$. By \cite[Lemma 4.2]{deeley2022}, $x$ and $y$ are synchronizing and by assumption $x\sim_\lcs y$. Hence, using Theorem \ref{sofic-lemma-resolving}, it follows that $a \sim_s b$. That is, $(a, b)\in G^s(\Sigma_G, q)$ and $(\pi \times \pi) (a, b)= (x,y)$ as required.
\end{proof}

\begin{remark}
The reader might wonder if the previous theorem generalizes to the case of finitely presented systems. This is indeed the case, but the proof is much more involved. It will appear (with other results) in \cite{DeeStoFinPre}
\end{remark}

\subsection{Short Exact Sequence, Morita Equivalence, and Structure}\label{sec:ses}

Suppose that $(X,\sigma)$ is a synchronizing shift space.  In \cite[Theorem 5.8]{deeley2022} we have shown there is a short exact sequence determined by the synchronizing ideal.  Since the above groupoids are all AF-groupoids, and hence amenable, this short exact sequence has the form
\[ 0 \longrightarrow \I_\sync(X, \sigma) \longrightarrow A(X,\sigma) \longrightarrow C^\ast\left(G^\lc_\text{non-sync}(X,\sigma)\right) \longrightarrow 0 \]
where $G^\lc_\text{non-sync}(X,\sigma) = G^\lc(X,\sigma)|_{X_\text{non-sync}}$.  In the case when $X$ is a mixing sofic shift, we also have that $\I_\sync(X,\sigma)$ is Morita equivalent to $S(X,\sigma,p) \otimes U(X,\sigma,p)$ for any choice of synchronizing periodic point $p \in X$, see \cite[Theorem 6.10]{deeley2022}. For a mixing sofic shift with finitely many non-synchronizing points, these results and Theorem \ref{sofic-min-right-alg} lead to a complete description of the $K$-theories of the relevant algebras:

\begin{corollary} \label{K-theory-mixing-sofic}
Suppose that $X$ is a mixing sofic shift, $(G_r, \pi_r)$ is its minimal right-resolving presentation, and $(G_l, \pi_l)$ is its minimal left-resolving presentation. Denote by $\Sigma_r$ and $\Sigma_l$ the shift of finite type associated to $G_r$ and $G_l$ respectively and fix synchronizing periodic points $p_r$ and $p_l$ in $X$ such that $\pi_r^{-1}(p_r)=\{ q_r \}$ and $\pi_l^{-1}(p_l)=\{ q_l \}$. Then 
\begin{enumerate}
\item the $K_1$-group of each algebra is trivial.
\item $K_0(S(X,p)) \cong K_0(S(\Sigma_r, q_r)) \cong $ the stationary inductive limit associated to $A_{G_r}$.
\item $K_0(U(X,p)) \cong K_0(U(\Sigma_l, q_l)) \cong $ the stationary inductive limit associated to $A_{G_l}$.
\item $K_0(\I_\sync(X))\cong K_0(S(X,p)) \otimes K_0(U(X,p))$.
\end{enumerate}
If moreover the set of non-synchronizing points of $X$ is finite, then 
\[ K_0(A(X)) \cong (K_0(S(X,p)) \otimes K_0(U(x,p))) \oplus \Z^n \] 
where $n$ is the number of non-synchronizing points.
\end{corollary}
\begin{proof}
The first item follows since all the relevant $C^\ast$-algebras are AF. The second and third follow from Theorem \ref{sofic-min-right-alg}; we note that different choices of synchronizing periodic points lead to Morita equivalence $C^\ast$-algebras, so the $K$-theory is unaffected. The fourth item follows from the Morita equivalence between $\I_\sync(X)$ and $S(X,p) \otimes U(x,p)$, the fact that AF-algebras have torsion-free $K$-theory, and the Kunneth formula. For the final part, we have short exact sequence 
\[ 0 \longrightarrow \I_\sync(X, \sigma) \longrightarrow A(X,\sigma) \longrightarrow \C^n \longrightarrow 0 \]
see \cite[Theorem 6.11]{deeley2022}. The result then follows from the previous parts and the six-term exact sequence in $K$-theory, which reduced to a split short exact sequence in this particular situation.
\end{proof}

\begin{corollary} \label{almost-finite-type-algebras}
Suppose that $X$ is a mixing sofic shift that is of almost finite type, $(G, \pi)$ is its minimal right and left resolving presentation, and $p$ is a periodic point of $X$ such that $\pi^{-1}(p)=\{ q \}$ for some $q\in \Sigma_G$. Then $S(X,p) \cong S(\Sigma, q)$, $U(X, p) \cong S(\Sigma, q)$, and $\I_\sync(X) \sim_{ME} A(\Sigma_G)$.
\end{corollary}
\begin{proof}
This follows from Theorem \ref{sofic-min-right-alg} and 
\[
\I_\sync(X) \sim_{ME} S(X, p)\otimes U(X, p) \cong S(\Sigma_G, q) \otimes U(\Sigma_G, q) \sim_{ME} A(\Sigma_G)
\]
\end{proof}

\section{The Even Shift}\label{sec:even-shift}

Recall from Example \ref{even-shift} that the even shift is a shift space in the alphabet $\{\word{0}, \word{1}\}$ whose elements don't contain any of the words in the set $F = \{\word{10}^{2k+1}\word{1} \mid k \geq 0\}$.  That is to say that the elements of the even shift are infinite binary sequences which have an even number of consecutive zeros between any two ones.  We will denote the even shift as $X_\ev$.

The even shift is a sofic shift of almost finite type with only one non-synchronizing point. Hence, we can apply Corollaries \ref{K-theory-mixing-sofic} and \ref{almost-finite-type-algebras} to obtain essentially the entire structure of the relevant algebras. However, we include the full details of the construction of the relevant Bratteli diagrams because it is illustrative and also there is a small mistake in the Bratteli diagram of the homoclinic algebra of the even shift given in \cite{thomsen2010c} and it seems appropriate to correct this.

\subsection{Homoclinic Algebra}

We will now construct the Bratteli diagram $B^\lc(X_\ev)$ as in Section \ref{sec:shift-bratteli-homoclinic}.  To this end, we must describe the context equivalence classes of words in $\lang(X_\ev)$.  If we assume that $n \geq 1$ and $w \neq \word{0}^{2n+1}$, then $w = \word{0}^k\word{1} \dots \word{1}\word{0}^l$ for some $k, l \geq 0$.  It is always the case in the even shift that the zero word of any length can be appended to either side of any word $w$, so we consider pairs $(a,b) \in \lang(X_\ev) \times \lang(X_\ev)$ of the form \[a=...\word{1}\word{0}^r \qquad b=\word{0}^s\word{1}...\]  In this case, $awb$ will be a word in the even shift if and only if $r+k$ and $l+s$ are both even.  Hence if $w$ and $v$ are both non-zero words of length $2n+1$ we have $E(w) = E(v)$ if and only if the parity of zeros on each side of $w$ and $v$ both match.  In particular, a word $w \in \lang(X_\ev)$ is synchronizing if and only if $w$ contains a $\word{1}$.  In the case of the word $\word{0}^{2n+1} \in \lang_{2n+1}(X_\ev)$, we can see that both $(\word{1},\word{0})$ and $(\word{0},\word{1})$ are in $E(\word{0}^{2n+1})$ but $(\word{1},\word{1})$ is not, which cannot be the case for any $w \neq \word{0}^{2n+1}$.  This means $E(w) \neq E(\word{0}^{2n+1})$ whenever $w \neq \word{0}^{2n+1}$.  Lastly, when $n=0$, $\lang_1(X_\ev) = \{\word{0},\word{1}\}$ and $E(\word{0}) \neq E(\word{1})$ so the words $\word{0}$ and $\word{1}$ each form their own equivalence class.

In summary, for $n = 0$ we have two context equivalence classes, and hence $V_0$ has two vertices in the Brattli diagram $B^\lc(X_\ev)$.  For $n \geq 1$, there are five context equivalence classes and so five vertices in $V_n$.  We then get edges from $V_n$ to $V_{n+1}$ as described in Section \ref{sec:shift-bratteli-homoclinic}.  For example there is an edge from $[\word{1}] \in V_0$ to $[\word{111}] = \{\word{111},\word{100},\word{001}\} \in V_1$ coming from the pair $(\word{1},\word{1}) \in E(\word{1})$.  For another example, there are two edges from $[\word{0}] \in V_0$ to $[\word{111}] \in V_1$ coming from the pairs $(\word{1},\word{0}), (\word{0},\word{1}) \in E(\word{0})$.  This situation is depicted in Figure \ref{fig:even-shift-lc-bratteli}.  The AF-algebra corresponding to $B^\lc(X_\ev)$ is isomorphic to $A(X_\ev,\sigma)$ by Theorem \ref{theorem:shift-bratteli-lc-isomorphism}.  We note that there is a slight discrepancy between the Bratteli diagram in Figure \ref{fig:even-shift-lc-bratteli} and the Bratteli diagram in \cite[Example 3.5]{thomsen2010c}, in particular the edges appear to be incorrect in Thomsen's diagram.

\begin{figure}
\centering

\begin{tikzpicture}[scale=0.85]
\pgfmathsetmacro{\W}{3};
\pgfmathsetmacro{\H}{5};
\pgfmathsetmacro{\D}{1pt};

\node[shape=circle, draw=black] (S) at (0,-1.25*\H) {};

\node[shape=rectangle, draw=black] (V00) at (-2*\W,-1.5*\H) {$\word{1}$};
\node[shape=rectangle, draw=black] (V01) at (2*\W,-1.5*\H) {$\word{0}$};
    
    \path (S) edge (V00);
    \path (S) edge (V01);

\node[shape=rectangle, draw=black, align=left] (V10) at (-2*\W,-2*\H) {$\word{111},\word{100},\word{001}$};
\node[shape=rectangle, draw=black] (V11) at (-1*\W,-2*\H) {$\word{110}$};
\node[shape=rectangle, draw=black] (V12) at (0,-2*\H) {$\word{011}$};
\node[shape=rectangle, draw=black] (V13) at (1*\W,-2*\H) {$\word{010}$};
\node[shape=rectangle, draw=black] (V14) at (2*\W,-2*\H) {$\word{000}$};

    \path (V00) edge (V10);
    \path (V00) edge (V11);
    \path (V00) edge (V12);
    \path (V00) edge (V13);
    
    \path (V01) edge[style={double,double distance=\D}] (V10);
    \path (V01) edge (V14);

\node[shape=rectangle, draw=black, align=left] (V20) at (-2*\W,-3*\H) {$\word{11111},\word{11001},$\\$\word{10011},\word{11100},$\\$\word{00111},\word{00100},$\\$\word{10000},\word{00001}$};
\node[shape=rectangle, draw=black, align=left] (V21) at (-1*\W,-3*\H) {$\word{11110},\word{11000},$\\$\word{10010},\word{00110}$};
\node[shape=rectangle, draw=black, align=left] (V22) at (0,-3*\H) {$\word{01111},\word{01001},$\\$\word{00011},\word{01100}$};
\node[shape=rectangle, draw=black, align=center] (V23) at (1*\W,-3*\H) {$\word{01110},\word{01000},$\\$\word{00010}$};
\node[shape=rectangle, draw=black] (V24) at (2*\W,-3*\H) {$\word{00000}$};
    
    \path (V10) edge (V20);
    \path (V10) edge (V21);
    \path (V10) edge (V22);
    \path (V10) edge (V23);
    
    \path (V11) edge (V20);
    \path (V11) edge (V22);
    
    \path (V12) edge (V20);
    \path (V12) edge (V21);
    
    \path (V13) edge (V20);
    
    \path (V14) edge[style={double,double distance=\D}] (V20);
    \path (V14) edge (V24);

\node[shape=rectangle, draw=black] (V30) at (-2*\W,-4*\H) {$\word{1111111},\dots$};
\node[shape=rectangle, draw=black] (V31) at (-1*\W,-4*\H) {$\word{1111110}, \dots$};
\node[shape=rectangle, draw=black] (V32) at (0,-4*\H) {$\word{0111111}, \dots$};
\node[shape=rectangle, draw=black] (V33) at (1*\W,-4*\H) {$\word{0111110}, \dots$};
\node[shape=rectangle, draw=black] (V34) at (2*\W,-4*\H) {$\word{0000000}$};
    
    \path (V20) edge (V30);
    \path (V20) edge (V31);
    \path (V20) edge (V32);
    \path (V20) edge (V33);
    
    \path (V21) edge (V30);
    \path (V21) edge (V32);
    
    \path (V22) edge (V30);
    \path (V22) edge (V31);
    
    \path (V23) edge (V30);
    
    \path (V24) edge[style={double, double distance=\D}] (V30);
    \path (V24) edge (V34);
    
\node (E0) at (-2*\W,-4.2*\H) {$\vdots$};
\node (E1) at (-1*\W,-4.2*\H) {$\vdots$};
\node (E2) at (0,-4.2*\H) {$\vdots$};
\node (E3) at (1*\W,-4.2*\H) {$\vdots$};
\node (E4) at (2*\W,-4.2*\H) {$\vdots$};
\end{tikzpicture}

\caption{The Bratteli diagram $B^\lc(X_\ev)$ for the even shift.  Paths in this Bratteli diagram indicate how to inductively build local conjugacy equivalence classes in the even shift.}
\label{fig:even-shift-lc-bratteli}
\end{figure}

As described in Section \ref{sec:shift-algebras}, we can compute the $K$-theory of $A(X_\ev,\sigma)$ directly from the Brattli diagram itself, \[ K_0\big(A(X_\ev,\sigma)\big) \cong \varinjlim  \left(\Z V_n, A_n \right) \] where $A_n$ are the transition matrices encoding the edge relation in $B^\lc(X_\ev)$.  In this case, $|V_n| = 5$ and $A_n$ is the matrix \[ A = \begin{bmatrix} 1 & 1 & 1 & 1 & 2 \\ 1 & 0 & 1 & 0 & 0 \\ 1 & 1 & 0 & 0 & 0 \\ 1 & 0 & 0 & 0 & 0 \\ 0 & 0 & 0 & 0 & 1 \end{bmatrix} \] for all $n \geq 1$.  Hence \[ K_0\big(A(X_\ev,\sigma)\big) \cong \varinjlim \left( \Z^5, A \right) \] Note that the direct limit can be computed as $\varinjlim \left( \Z^5, A \right) \cong \big( \Z^5 \times \N \big) / \sim$ where, for $n \leq k$, $(v,n) \sim (w,k)$ if and only if $vA^{k-n+l} = wA^l$ for some $l \geq 0$.

\subsection{Heteroclinic Synchronizing Algebras and Synchronizing Ideal}

Notice that in the computation of $K_0\big(A(X_\ev,\sigma)\big)$, if we restrict to the synchronizing vertices we get the transition matrix \[ A_\sync = \begin{bmatrix} 1 & 1 & 1 & 1 \\ 1 & 0 & 1 & 0 \\ 1 & 1 & 0 & 0 \\ 1 & 0 & 0 & 0 \end{bmatrix} \,. \]  Consequently, we obtain \[ K_0 \big( \I_\sync(X_\ev, \sigma) \big) \cong \varinjlim \left( \Z^4, A_\sync \right) \,. \]

Next we will compute the Bratteli diagrams $B^\lcs_0(X_\ev,p)$ and $B^\lcu_0(X_\ev,p)$ where $p = \overline{\word{1}}$ is the sequence $\overline{\word{1}} = \ldots \word{11.111} \ldots$, which is a synchronizing fixed point in the even shift.  Observe that the follower set of $p_{(-\infty,0]} = \ldots \word{111}$ is the same as $E^+(\word{1})$.  Hence in $B^\lcs_0(X_\ev,p)$, we only have two equivalence classes of words in $V_n$ for each $n \geq 0$ since $E^+(\word{1}w)$ is determined by the parity of zeros on the right side of $\word{1}w$.  This Bratteli diagram, and additionally $B^\lcu_0(X_\ev,p)$, are depicted in Figure \ref{fig:even-bratteli-heteroclinic}.

\begin{figure}
\centering

\begin{tikzpicture}[scale=0.7]
\pgfmathsetmacro{\W}{3};
\pgfmathsetmacro{\H}{5};
\pgfmathsetmacro{\D}{1pt};

\node[align=center] at (-1.5*\W, -0.95*\H) {$B^\lcs_0\left(X_\text{even}, \overline{\word{1}}\right)$};

\node[shape=rectangle, draw=black] (S) at (-1.5*\W,-1.25*\H) {$\ldots \word{111}$};

\node[shape=rectangle, draw=black] (V00) at (-2.25*\W,-1.5*\H) {$\word{1}$};
\node[shape=rectangle, draw=black] (V01) at (-0.75*\W,-1.5*\H) {$\word{0}$};
    
    \path (S) edge (V00);
    \path (S) edge (V01);

\node[shape=rectangle, draw=black, align=left] (V10) at (-2.25*\W,-2*\H) {$\word{11},\word{00}$};
\node[shape=rectangle, draw=black] (V11) at (-0.75*\W,-2*\H) {$\word{10}$};

    \path (V00) edge (V10);
    \path (V00) edge (V11);
    
    \path (V01) edge (V10);

\node[shape=rectangle, draw=black, align=left] (V20) at (-2.25*\W,-2.5*\H) {$\word{111},\word{001},\word{100}$};
\node[shape=rectangle, draw=black, align=left] (V21) at (-0.75*\W,-2.5*\H) {$\word{110},\word{000}$};
    
    \path (V10) edge (V20);
    \path (V10) edge (V21);
    
    \path (V11) edge (V20);

\node[shape=rectangle, draw=black, align=left] (V30) at (-2.25*\W,-3.1*\H) {$\word{1111},\word{0011},$\\$\word{1001},\word{1100},$\\$\word{0000}$};
\node[shape=rectangle, draw=black, align=left] (V31) at (-0.75*\W,-3.1*\H) {$\word{1110},\word{0010},$\\$\word{1000}$};
    
    \path (V20) edge (V30);
    \path (V20) edge (V31);
    
    \path (V21) edge (V30);
    
\node (E0) at (-2.25*\W,-3.4*\H) {$\vdots$};
\node (E1) at (-0.75*\W,-3.4*\H) {$\vdots$};

\node[align=center] at (1.5*\W, -0.95*\H) {$B^\lcu_0\left(X_\text{even}, \overline{\word{1}}\right)$};

\node[shape=rectangle, draw=black] (S) at (1.5*\W,-1.25*\H) {$\word{111} \ldots$};

\node[shape=rectangle, draw=black] (V00) at (0.75*\W,-1.5*\H) {$\word{1}$};
\node[shape=rectangle, draw=black] (V01) at (2.25*\W,-1.5*\H) {$\word{0}$};
    
    \path (S) edge (V00);
    \path (S) edge (V01);

\node[shape=rectangle, draw=black, align=left] (V10) at (0.75*\W,-2*\H) {$\word{11},\word{00}$};
\node[shape=rectangle, draw=black] (V11) at (2.25*\W,-2*\H) {$\word{01}$};

    \path (V00) edge (V10);
    \path (V00) edge (V11);
    
    \path (V01) edge (V10);

\node[shape=rectangle, draw=black, align=left] (V20) at (0.75*\W,-2.5*\H) {$\word{111},\word{001},\word{100}$};
\node[shape=rectangle, draw=black, align=left] (V21) at (2.25*\W,-2.5*\H) {$\word{011},\word{000}$};
    
    \path (V10) edge (V20);
    \path (V10) edge (V21);
    
    \path (V11) edge (V20);

\node[shape=rectangle, draw=black, align=left] (V30) at (0.75*\W,-3.1*\H) {$\word{1111},\word{0011},$\\$\word{1001},\word{1100},$\\$\word{0000}$};
\node[shape=rectangle, draw=black, align=left] (V31) at (2.25*\W,-3.1*\H) {$\word{0111},\word{0001},$\\$\word{1000}$};
    
    \path (V20) edge (V30);
    \path (V20) edge (V31);
    
    \path (V21) edge (V30);
    
\node (E0) at (0.75*\W,-3.4*\H) {$\vdots$};
\node (E1) at (2.25*\W,-3.4*\H) {$\vdots$};
\end{tikzpicture}

\caption{The Bratteli diagrams $B^\lcs_0(X_\text{even}, \overline{\word{1}})$ and $B^\lcu_0(X_\text{even}, \overline{\word{1}})$ for the even shift.  Here we are using the synchronizing periodic point $p = \overline{\word{1}} = \ldots \word{11.111} \ldots$.  The top vertex in each diagram indicates the left-infinite ray $\overline{\word{1}}_{(-\infty,0]} = \ldots \word{111}$ and the right-infinite ray $\overline{\word{1}}_{[0,\infty)} = \word{111} \ldots$, respectively.  Then, paths in each diagram indicate how to build stable (unstable) local conjugacy equivalence classes inductively.}
\label{fig:even-bratteli-heteroclinic}
\end{figure}

Since $X_\ev$ is a mixing sofic shift, by Theorem \ref{theorem:shift-bratteli-lcu-lcs-morita-equivalence} we can compute the $K$-theory of $S(X_\ev, \sigma, p)$ and $U(X_\ev, \sigma, p)$ from the Bratteli diagrams $B^\lcs_0(X_\ev,p)$ and $B^\lcu_0(X_\ev, p)$ respectively.  From each of the Bratteli diagrams $B^\lcs_0(X_\ev,p)$ and $B^\lcu_0(X_\ev,p)$, we can see that the transition matrices are given by \[ A_\lcs = A_\lcu = \begin{bmatrix} 1 & 1 \\ 1 & 0 \end{bmatrix} \,. \]  Hence the $K$-theory of $S(X_\ev, \sigma, p)$ and $U(X_\ev, \sigma, p)$ can be computed as
\begin{align*}
    K_0\big(S(X_\ev, \sigma, p)\big) &\cong \varinjlim \left( \Z^2, A_\lcs \right) \text{ , and } \\
    K_0\big(U(X_\ev, \sigma, p)\big) &\cong \varinjlim \left( \Z^2, A_\lcu \right) \,.
\end{align*}
We also remark that by the Morita equivalence in \cite[Theorem 5.5]{deeley2022}, the $K_0$ groups computed above are independent of our choice of $p$.

Lastly we point out that \[ A_\lcs \otimes A_\lcu = \begin{bmatrix} 1 & 1 & 1 & 1 \\ 1 & 0 & 1 & 0 \\ 1 & 1 & 0 & 0 \\ 1 & 0 & 0 & 0 \end{bmatrix} = A_\sync \,. \]  This is a direct reflection of the fact from \cite[Theorem 6.11]{deeley2022} that $S(X_\ev, \sigma, p) \otimes U(X_\ev, \sigma, p)$ is Morita equivalent to $\I_\sync(X_\ev, \sigma)$, which follows from that fact that the even shift is mixing and the fact that sofic shifts are finitely presented.

In particular, observe that since $A$ and $A_\sync$ are both invertible we have \[ \text{rank}(K_0(A(X_\ev, \sigma))) = 5 \] while \[ \text{rank}(K_0(\I_\sync(X_\ev, \sigma))) = 4 \,.\]  Since $S(X_\ev, \sigma, p) \otimes U(X_\ev, \sigma, p)$ is Morita equivalent to $\I_\sync(X_\ev, \sigma)$ and $K$-theory is invariant under Morita equivalence, we conclude that $A(X_\ev, \sigma)$ cannot be Morita equivalent to $S(X_\ev, \sigma, p) \otimes U(X_\ev, \sigma, p)$.  Since this fact is true for Smale spaces (see \cite{putnam_1996}), and consequently shifts of finite type, we can see that the failure of $X_\ev$ to be a shift of finite type can be shown from the $K$-theory of its homoclinic and heteroclinic algebras.

Note that, as we remarked before, the heteroclinic synchronizing algebras are distinct from the heteroclinic algebras defined by Thomsen.  This can be seen in the fact that our heteroclinic synchronizing algebras of the even shift do not pick up on the point of all zeros, unlike the heteroclinic algebras compute in \cite[Example 3.5]{thomsen2010c}.

\section{The Charge Constrained Shift}\label{sec:dc-shift}

Fix a positive integer $c$.  Then the \emph{charge constrained shift} $X_c \subseteq \{ \word{+1}, -1 \}^\Z$ is the set of all sequences $x \in \{\word{+1},\word{-1}\}^\Z$ such that the sum $s$ of the entries of any finite subword of $x$ satisfies $-c \leq s \leq c$.  This is a sofic shift for all $c$, see \cite[Example 1.2.7]{lindmarcus}.  The charge constrained shift admits the following labeled graph presentation with $c+1$ vertices, which we denote $G_c$.  Note that $G_c$ is an irreducible graph, so $X_c$ is irreducible \cite{lindmarcus}.

\vspace{1em}
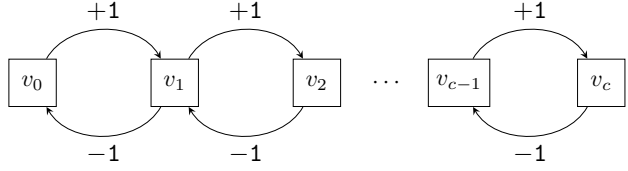
\begin{center}
\scalebox{0.9}{
\begin{tikzpicture}[scale=0.7]
\tikzset{every loop/.style={looseness=10, in=130, out=230}}
\pgfmathsetmacro{\S}{1.1}
\pgfmathsetmacro{\D}{0.7cm}

\node[draw=black, minimum size=\D] (A) at (0,0) {$v_0$};
\node[draw=black, minimum size=\D] (B) at (3,0) {$v_1$};
\node[draw=black, minimum size=\D] (C) at (6,0) {$v_2$};
\node[draw=black, minimum size=\D] (D) at (9,0) {$v_{c-1}$};
\node[draw=black, minimum size=\D] (E) at (12,0) {$v_{c}$};

\path[->,>=stealth] (A) edge[out=60, in=120] node[scale=\S, above] {$\word{+1}$} (B);
\path[->,>=stealth] (B) edge[out=240, in=300] node[scale=\S, below] {$\word{-1}$} (A);
\path[->,>=stealth] (B) edge[out=60, in=120] node[scale=\S, above] {$\word{+1}$} (C);
\path[->,>=stealth] (C) edge[out=240, in=300] node[scale=\S, below] {$\word{-1}$} (B);
\path[->,>=stealth] (D) edge[out=60, in=120] node[scale=\S, above] {$\word{+1}$} (E);
\path[->,>=stealth] (E) edge[out=240, in=300] node[scale=\S, below] {$\word{-1}$} (D);
\node[align=center] at (7.5, 0) {$\cdots$};
\end{tikzpicture}
}
\end{center}
Note that there are two different ways to embed the graph $G_{c-1}$ into the graph $G_c$.  Let $\{w_0, w_1, \dots, w_{c - 1} \}$ denote the vertices of the graph $G_{c-1}$, then we can embed $G_{c-1}$ into $G_c$ via
\begin{align*}
    w_i &\mapsto v_i \text{ , or } \\
    w_i &\mapsto v_{i+1}
\end{align*}
for $0 \leq i < c$.  Hence we get an embedding of $X_{c-1}$ into $X_c$.  From another perspective, we get this embedding because for every $x \in X_{c-1}$ the sum $s$ of the entries of any finite subword of $x$ satisfies $-c < -c+1 \leq s \leq c-1 < c$.

The primary intention of introducing this particular shift space is that it is a synchronizing shift for which the set of non-synchronizing points, $X_c \setminus (X_c)_\sync$, is uncountable.  The charge constrained shift is synchronizing because it is formed from a finite labeled graph, hence it is sofic.  Furthermore, $X_c$ is a shift space which is \emph{almost of finite type (AFT)}, see \cite{boyle85}.  The graph $G_c$ is indeed the minimal presentation of $X_c$ which is both left-resolving and right-resolving.  The even shift from Example \ref{even-shift} is also a shift space which is almost of finite type.

\begin{theorem}
For $c > 2$, the set of non-synchronizing points in the shift space $X_c$ is uncountable.
\end{theorem}
\begin{proof}
Fix $c > 2$.  If $x \in X_{c-1} \subseteq X_c$ then there will always be two distinct bi-infinite paths in $G_c$ corresponding to $x$ under the factor map $\pi:X_{G_c} \to X_c$.  This is because one can see there are two ways to embed the graph $G_{c-1}$ into $G_c$.  Assume for the sake of contradiction that $x \in X_{c-1} \subseteq X_c$ and there is an $n$ such that $x_{[-n,n]}$ is a synchronizing word.  Let $p_1$ and $p_2$ be two distinct paths on $G_c$ which represent the word $x_{[-n,n]}$.  By the two embeddings of $G_{c-1}$ described above, without loss of generality we have
\begin{align*}
    &s(p_1) = v_i &&s(p_2) = v_{i+1} \\
    &r(p_1) = v_j &&r(p_2) = v_{j+1}
\end{align*}
for some $0 \leq i, j < c$.  In particular, the sum of the entries in $x_{[-n,n]}$ is $j - i$.  Observe that then we have both
\begin{align*}
    \Big((\word{-1})(\word{+1})^i, (\word{+1})^{c-j}(\word{-1}) \Big) &\in E\parens{x_{[-n,n]}} \text{ , and } \\
    \Big((\word{-1})(\word{+1})^{i+1}, (\word{+1})^{c-j-1}(\word{-1}) \Big) &\in E\parens{x_{[-n,n]}}
\end{align*}
from the existence of $p_1$ and $p_2$ respectively.  Since we are assuming $x_{[-n,n]}$ is synchronizing, $\displaystyle \Big((\word{-1})(\word{+1})^{i+1}, (\word{+1})^{c-j}(\word{-1})\Big) \in E\parens{x_{[-n,n]}}$.  Thus we can form the word \[ w = (\word{+1})^{i+1} x_{[-n,n]} (\word{+1})^{c-j} \,. \]  However, the sum of the entries in $w$ is $(i + 1) + (j - i) + (c - j) = c + 1$. This is a contradiction since the sum of the entries in $w$ exceeds $c$, hence $w$ is not a valid word in $X_c$.  Hence if $x \in X_{c-1} \subseteq X_c$ then $x$ cannot be synchronizing.  

Observe that $X_c$ is uncountable for $c \geq 2$.  A quick proof of this follows from the fact that $X_2 \subseteq X_c$ for $c \geq 2$, and from the fact that there is an injective function from the uncountable set $\{ \word{a}, \word{b} \}^\Z$ to $X_2$ given by $\word{a} \mapsto (\word{+1})(\word{-1})$ and $\word{b} \mapsto (\word{-1})(\word{+1})$.  Since we are assuming $c > 2$ we thus have that $X_{c-1}$ is uncountable.  In conclusion, we have argued that $X_{c-1} \subseteq X_c \setminus (X_c)_\sync$, and it follows that the set $X_c \setminus (X_c)_\sync$ is uncountable.
\end{proof}

Even though the set of non-synchronizing points is infinite, the heteroclinic synchronizing algebras and the synchronizing ideal are easy to determine. This follows since $X_c$ is of almost finite type and we have its minimal presentation. Explicitly, for example when $c=4$, one get that $S(X_c) \cong U(X_c) \cong$ the unique $C^*$-stable AF-algebra whose $K_0$-group is obtained from the stationary inductive limit with the matrix 
\[
\begin{bmatrix} 0 & 1 & 0 & 0 \\ 1 & 0 & 1 & 0  \\ 0 & 1 & 0 & 1 \\ 0 & 0 & 1 & 0 \end{bmatrix}
\]
The $K$-theory of the synchronizing ideal is then obtained via Corollary \ref{K-theory-mixing-sofic}.

\section{Sofic shifts that are not almost finite type} \label{sec:sofic-not-finite}

In this section, we consider an example of a sofic shift where we can see that it is not of almost finite type using $K$-theory. To be clear, there are much more effective way of checking if a sofic shift is of almost finite type and indeed to compute the $K$-theory one (often) needs the minimal right and left resolving covers. However, it is of interest that in some examples the $K$-theory retains this information.

We consider the sofic shift associated to the following labelled graph:

\vspace{1em}
\begin{center}
\begin{tikzpicture}
\tikzset{every loop/.style={looseness=10, in=130, out=230}}
\pgfmathsetmacro{\S}{0.9}

\node[shape=circle, draw=black] (A) at (0,0) {};
\node[shape=circle, draw=black] (B) at (2,0) {};
\path[->,>=stealth] (A) edge[in=200, out=250, looseness=10] node[scale=\S, left] {$\word{c}$} (A); 
\path[->,>=stealth] (A) edge[in=100, out=150, looseness=10] node[scale=\S, left] {$\word{a}$} (A); 
\path[->,>=stealth] (B) edge[in=300, out=350, looseness=10] node[scale=\S, right] {$\word{a}$} (B); 
\path[->,>=stealth] (B) edge[in=370, out=420, looseness=10] node[scale=\S, right] {$\word{f}$} (B);
\path[->,>=stealth] (A) edge[out=30, in=150] node[scale=\S, above] {$\word{b}$} (B);
\path[->,>=stealth] (A) edge[out=70, in=110] node[scale=\S, above] {$\word{d}$} (B);
\path[->,>=stealth] (B) edge[out=230, in=310] node[scale=\S, above] {$\word{b}$} (A);
\path[->,>=stealth] (B) edge[out=260, in=280] node[scale=\S, below] {$\word{c}$} (A);
\end{tikzpicture}    
\end{center}

We learn of this sofic shift from the slides of a talk given by Soren Eilers. The presentation above is the minimal right-resolving presentation. The minimal left-resolving presentation is given by the labelled graph:

\vspace{1em}
\begin{center}
\begin{tikzpicture}
\tikzset{every loop/.style={looseness=10, in=130, out=230}}
\pgfmathsetmacro{\S}{0.9}

\node[shape=circle, draw=black] (A) at (-2,0) {};
\node[shape=circle, draw=black] (B) at (2,0) {};
\node[shape=circle, draw=black] (C) at (0, -2) {};
\path[->,>=stealth] (A) edge[in=100, out=150, looseness=10] node[scale=\S, left] {$\word{a}$} (A);

\path[->,>=stealth] (B) edge[in=300, out=350, looseness=10] node[scale=\S, right] {$\word{a}$} (B); 
\path[->,>=stealth] (B) edge[in=370, out=420, looseness=10] node[scale=\S, right] {$\word{f}$} (B);

\path[->,>=stealth] (C) edge[in=200, out=240, looseness=10] node[scale=\S, left] {$\word{a}$} (C);
\path[->,>=stealth] (C) edge[in=250, out=290, looseness=10] node[scale=\S, below] {$\word{b}$} (C); 
\path[->,>=stealth] (C) edge[in=300, out=340, looseness=10] node[scale=\S, right] {$\word{c}$} (C);

\path[->,>=stealth] (A) edge[out=30, in=150] node[scale=\S, above] {$\word{b}$} (B);
\path[->,>=stealth] (A) edge[out=70, in=110] node[scale=\S, above] {$\word{d}$} (B);
\path[->,>=stealth] (A) edge[out=330, in=120] node[scale=\S, above] {$\word{d}$} (C);

\path[->,>=stealth] (B) edge[out=190, in=350] node[scale=\S, above] {$\word{b}$} (A);
\path[->,>=stealth] (B) edge[out=250, in=380] node[scale=\S, above] {$\word{f}$} (C);

\path[->,>=stealth] (C) edge[out=160, in=300] node[scale=\S, above] {$\word{c}$} (A);

\end{tikzpicture}    
\end{center}

The adjacency matrix of these two graphs are respectively 
\[
\begin{bmatrix} 2 & 2  \\  2 & 2 \end{bmatrix} \hbox{ and }\begin{bmatrix} 1 & 2 & 1  \\ 1 & 2 & 1  \\ 1 & 0 & 3  \end{bmatrix}
\]
The rank of the first of these matrices is one and the second has rank two. It follows from Corollary \ref{K-theory-mixing-sofic} that $K_0(\I_{sync})$ has rank two, but for any shift of finite type the $K_0$-group of the homoclinic algebra is a perfect square (since $K_0(S)$ and $K_0(U)$ have the same rank for shifts of finite type). It follows from Corollary \ref{almost-finite-type-algebras} (among other ways) that this sofic shift is not of almost finite type.

\section{The $\word{a}^n\word{b}^n$-Shift}\label{context-free-shift}

What we call the $\word{a}^n\word{b}^n$-shift is a modified version of what is called the context-free shift in \cite{lindmarcus}.  We let the \emph{$\word{a}^n\word{b}^n$-shift} be the closure of the set of sequences in $\{ \word{a}, \word{b} \}^\mathbb{Z}$ where any finite string of $\word{a}$'s, if it is followed by a string of $\word{b}$'s, is followed by a finite string of $\word{b}$'s of the same length.  We denote this shift space $X_\word{ab} \subseteq \{\word{a}, \word{b}\}^\Z$.  Alternatively one can think of this as the closure of the set of bi-infinite sequences formed from freely concatenating strings in the set $\{\word{a}^n\word{b}^n \mid n \geq 1\}$.  From this latter fact, the $\word{a}^n\word{b}^n$-shift is by definition a \emph{coded system} \cite{fiebig91}.  However, it is in particular a synchronizing shift.

\begin{lemma}\label{lem:cf-shift-is-sync}
The $\word{a}^n\word{b}^n$-shift is a synchronizing shift.
\end{lemma}
\begin{proof}
We will show that the word $\word{ba} \in \lang(X_\word{ab})$ is a synchronizing word.  Any word that can be added on the left of $\word{ba}$ is essentially of the form $w \word{b} \word{a}^n \word{b}^{n-1}$, and word that can be added on the right is of the form $\word{a}^{m-1} \word{b}^m \word{a} v$ --- where $v$ and $w$ are arbitrary.  Observe that $w \word{b} \word{a}^n \word{b}^n \word{a}^m \word{b}^m \word{a} v$ is a valid word in the $\word{a}^n\word{b}^n$-shift, hence $\word{ba}$ is a synchronizing word.  The irreducibility of the $\word{a}^n\word{b}^n$-shift can be seen from considering the case when $w = w' \word{b}\word{a}^{n+k} \word{b}^n$ and $v = \word{a}^m \word{b}^{m + l} \word{a} v'$, then $w \word{b}^k \word{a}^l v$ is a valid word in $X_\word{ab}$.
\end{proof}

\begin{lemma}\label{lem:cf-sync-iff-ba}
Suppose $w \in \lang(X_\word{ab})$ is a word in the $\word{a}^n\word{b}^n$-shift, then $w$ is synchronizing if and only if $\word{ba} \sqsubseteq w$.
\end{lemma}
\begin{proof}
Since $\word{ba}$ is synchronizing, as shown in the proof of Lemma \ref{lem:cf-shift-is-sync}, $\word{ba} \sqsubseteq w$ implies $w$ is synchronizing.  Conversely, assume $\word{ba} \not\sqsubseteq w$, then $w = \word{a}^i\word{b}^j$ for some $i,j \geq 1$ (the words $\word{a}$, $\word{b}$, and the empty word, are not synchronizing).  Then, without loss of generality, if $i \leq j$ we have $(\word{b}\word{a}^{j - i}, \word{a}), (\word{b}\word{a}^{j - i + 1}, \word{ba}) \in E(w)$ but $(\word{b}\word{a}^{j - i}, \word{ba}) \not\in E(w)$, hence $E(w)$ is not a cartesian product and $w$ is not synchronizing.
\end{proof}

\subsection{Graph Presentation}\label{sec:cf-entropy}

In this section we discuss some properties of a graph with countable vertices which we will eventually give a labeling in the next section.  Consider the infinite graph \( G \) shown in the diagram below.

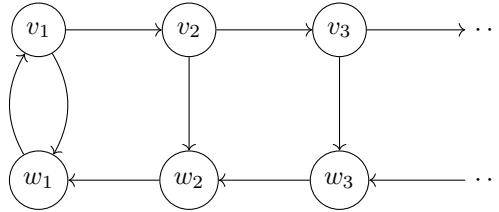
\begin{center}
\begin{tikzpicture}[scale=2]

\tikzset{vertex/.style = {shape=circle, draw}}
\tikzset{dots/.style = {}}
\tikzset{edge/.style = {->}}

\node[vertex] (v1) at (0, 1) {\(v_1\)};
\node[vertex] (v2) at (1, 1) {\(v_2\)};
\node[vertex] (v3) at (2, 1) {\(v_3\)};
\node[dots] (v4) at (3, 1) {\(\cdots\)};
\node[vertex] (w1) at (0, 0) {\(w_1\)};
\node[vertex] (w2) at (1, 0) {\(w_2\)};
\node[vertex] (w3) at (2, 0) {\(w_3\)};
\node[dots] (w4) at (3, 0) {\(\cdots\)};

\draw[edge] (v1) to[bend left] (w1);
\draw[edge] (v1) to (v2);
\draw[edge] (v2) to (w2);
\draw[edge] (v2) to (v3);
\draw[edge] (v3) to (w3);
\draw[edge] (v3) to (v4);
\draw[edge] (w1) to[bend left] (v1);
\draw[edge] (w2) to (w1);
\draw[edge] (w3) to (w2);
\draw[edge] (w4) to (w3);

\end{tikzpicture}
\end{center}
We can approximate \( G \) by a sequence of finite graphs \( \{ G_n \}_{n \geq 1} \) where \( G_n \) is the subgraph of \( G \) obtained by restricting to the set of vertices \( \{v_1, \dots, v_n, w_1, \dots, w_n \} \).  The adjacency matrix \( A_n \) corresponding to \( G_n \) can be described recursively by
\begin{align*}
    A_1 &= \begin{bmatrix} 0 & 1 \\ 1 & 0 \end{bmatrix} \in M_2(\mathbb{N}) \\
    A_{n+1} &= \begin{bmatrix} A_n & \begin{matrix} 0 & 0 \\ \vdots & \vdots \\ 0 & 0 \\ 0 & 0 \\ 0 & 1 \end{matrix} \\ \begin{matrix} 0 & \cdots & 0 & 1 & 0 \\ 0 & \cdots & 0 & 0 & 0 \end{matrix} & \begin{matrix} 0 & 0 \\ 1 & 0 \end{matrix} \end{bmatrix} \in M_{2n + 2}(\mathbb{N})
\end{align*}
Because each \( G_n \) is an irreducible graph, each \( A_n \) has a unique positive real eigenvalue equal to its spectral radius, called the Perron-Frobenius eigenvalue of \( A_n \), which we will denote as \( \lambda_n \).  We will compute \[ \lambda = \lim_{n \to \infty} \lambda_n \,. \]
We must first compute the characteristic polynomial \( p_n(t) \) for each \( A_n \).
\begin{lemma}
For all \( n \geq 1 \) we have the following relation.
\[ \det\left( \begin{matrix} tI_{2n} - A_n & \begin{matrix} 0 \\ \vdots \\ 0 \\ -1 \end{matrix} \\ \begin{matrix} 0 & \cdots & 0 & -1 & 0 \end{matrix} & 0 \end{matrix} \right) = -1 \]
\end{lemma}
\begin{proof}
In the case of \( n = 1 \) it is easy to check this is true.  Now assuming the determinant in the statement of the lemma is true for some \( n \geq 1 \), we compute the following using cofactor expansions
\begin{align*}
    \det\left( \begin{matrix} tI_{2n + 2} - A_{n + 1} & \begin{matrix} 0 \\ \vdots \\ 0 \\ -1 \end{matrix} \\ \begin{matrix} 0 & \cdots & 0 & -1 & 0 \end{matrix} & 0 \end{matrix} \right) &= \det\left( \begin{matrix} tI_{2n} - A_n & \begin{matrix} 0 & 0 & 0 \\ \vdots & \vdots & \vdots \\ 0 & 0 & 0 \\ 0 & -1 & 0 \end{matrix} \\ \begin{matrix} 0 & \cdots & 0 & -1 & 0 \\ 0 & \cdots & 0 & 0 & 0 \\ 0 & \cdots & 0 & 0 & 0 \end{matrix} & \begin{matrix} t & 0 & 0 \\ -1 & t & -1 \\ -1 & 0 & 0 \end{matrix} \end{matrix} \right) \\
    &= - \det\left( \begin{matrix} tI_{2n} - A_n & \begin{matrix} 0 & 0 \\ \vdots & \vdots \\ 0 & 0 \\ -1 & 0 \end{matrix} \\ \begin{matrix} 0 & \cdots & 0 & -1 & 0 \\ 0 & \cdots & 0 & 0 & 0 \end{matrix} & \begin{matrix} 0 & 0 \\ t & -1 \end{matrix} \end{matrix} \right) \\
    &= \det\left( \begin{matrix} tI_{2n} - A_n & \begin{matrix} 0 \\ \vdots \\ 0 \\ -1 \end{matrix} \\ \begin{matrix} 0 & \cdots & 0 & -1 & 0 \end{matrix} & 0 \end{matrix} \right) \\
    &= -1 \,.
\end{align*}
\end{proof}

\begin{lemma}\label{lem:cf-shift-transition-matrix-invertible}
For each \( n \geq 1 \), the characteristic polynomial \( p_n(t) \) of the matrix \( A_n \) satisfies \[ p_{n+1}(t) = t^2 p_n(t) - 1 \] where \( p_1(t) = t^2 - 1 \).
\end{lemma}
\begin{proof}
That \( p_1(t) = t^2 - 1 \) is clear.  Then we compute via cofactor expansion
\begin{align*}
    p_{n+1}(t) &= \det\left(tI_{2n+2} - A_{n+1}\right) \\
    &= \det\left( \begin{matrix} tI_{2n} - A_n & \begin{matrix} 0 & 0 \\ \vdots & \vdots \\ 0 & 0 \\ 0 & -1 \end{matrix} \\ \begin{matrix} 0 & \cdots & 0 & -1 & 0 \\ 0 & \cdots & 0 & 0 & 0 \end{matrix} & \begin{matrix} t & 0 \\ -1 & t \end{matrix} \end{matrix} \right) \\
    &= t\det\left( \begin{matrix} tI_{2n} - A_n & \begin{matrix} 0 \\ \vdots \\ 0 \\ -1 \end{matrix} \\ \begin{matrix} 0 & \cdots & 0 & 0 & 0 \end{matrix} & t \end{matrix} \right) + \det\left( \begin{matrix} tI_{2n} - A_n & \begin{matrix} 0 \\ \vdots \\ 0 \\ -1 \end{matrix} \\ \begin{matrix} 0 & \cdots & 0 & -1 & 0 \end{matrix} & 0 \end{matrix} \right) \\
    &= t^2 \det\left(tI_{2n} - A_n \right) - 1 \\
    &= t^2 p_n(t) - 1 \,.
\end{align*}
\end{proof}

From the recursive definition of the characteristic polynomials of \( A_n \) we find that \[ p_n(t) = t^{2n} - \sum_{k = 0}^{n-1} t^{2k} \,. \]  In particular, $\det(A_n) = -1$ for each $n$, so each $A_n$ is invertible over $\Z$.

Towards computing $\lambda$, we define \( q_n(t) = (t^2 - 1)p_n(t) = t^{2n + 2} - 2t^{2n} + 1 \).  Note that \( \lambda_n \) is a solution to \( q_n(t) \) for each \( n \).  Observe that the largest zero of \( q'_n(t) \) is \(\displaystyle t_n = \sqrt{\frac{2n}{n+1}} \) and that \( q'_n(t) > 0 \) for \( t > t_n \).  Furthermore one can check that \( q_n(t_n) \leq 0 \) and \( q_n(\sqrt{2}) = 1 \) for each \( n \), so since \( \lambda_n \) is the largest real root of \( q_n(t) \) we must have that \[ \sqrt{\frac{2n}{n+1}} \leq \lambda_n < \sqrt{2} \] for all \( n \).  Taking the limit with respect to \( n \), we obtain \[ \lambda = \lim_{n \to \infty} \lambda_n = \sqrt{2} \,. \] In the next section we will see that a labeling of the edges in $G$ gives a presentation of the $\word{a}^n\word{b}^n$-shift.  We conjecture that the topological entropy of $X_\word{ab}$, defined to be \[ h(X_\word{ab}) = \lim_{n \to \infty} \frac{1}{n} \log\left(|\lang_n(X_\word{ab})|\right) \,, \] is equal to $\log(\lambda)$.  However this is not in general the case for synchronizing shifts, see \cite{petersen86}.

\subsection{$K$-Theory is Infinite Rank}\label{sec:infinite-rank}

We will show that the $K_0$ group of the homoclinic algebra of the $\word{a}^n\word{b}^n$-shift is infinite rank.  This is in contrast to sofic shifts which have finite rank $K$-theory.  This latter fact follows from \cite[Theorem 3.2.10]{lindmarcus} which says that sofic shifts have a finite number of follower sets, and from the construction in Section \ref{sec:shift-bratteli-homoclinic}.  We can compute the $K$-theory of a shift space $X$ from the Bratteli diagram $B^\lc(X)$ --- which is itself built from follower sets.

In the case of the homoclinic algebra $A(X_\word{ab}, \sigma)$ of the $\word{a}^n\word{b}^n$-shift, we will show that $K_0\big(A(X_\word{ab}, \sigma)\big)$ is infinite rank.  Note from Section \ref{sec:shift-bratteli-homoclinic} that $K_0\big(A(X_\word{ab}, \sigma)\big)$ can be computed as
\[ K_0\big(A(X_\word{ab}, \sigma)\big) = \varinjlim \left( \Z^{N_n}, \widetilde{A}_n \right) \]
where $\widetilde{A}_n$ are the transition matrices in the Bratteli diagram $B^\lc(X_\word{ab})$.  We will denote these transition matrices as $\widetilde{A}_n$ in order to disambiguate them from the adjacency matrices $A_n$ in Section \ref{sec:cf-entropy}.

Consider the two graphs in Figure \ref{fig:context-free-shift-graph}.  Notice that both $G_r$ and $G_l$ are labeled versions of the graph $G$ in the previous section.  The important aspect of these graphs is that $G_r$ and $G_l$ are, respectively, \emph{right-resolving} and \emph{left-resolving}, see \cite[Section 3.3]{lindmarcus}.  For $G_r$ (right-resolving), this means that for each vertex there is at most a single edge out of that vertex with a given label.  Likewise for $G_l$ (left-resolving), each vertex has at most a single edge into that vertex with a given label.  Also see \cite[Definition 0.12]{fiebig91} for more details.

Recall from Lemma \ref{lem:cf-sync-iff-ba} that a word $w$ in the $\word{a}^n\word{b}^n$-shift is synchronizing if and only if $w$ contains the word $\word{ba}$ as a subword.  Consequently, a path which represents a synchronizing word must thus go through $s_r$ on $G_r$ or $s_l$ on $G_l$.

\begin{figure}
\centering

\begin{tikzpicture}[scale=2]

\tikzset{vertex/.style = {shape=circle, draw}}
\tikzset{dots/.style = {}}
\tikzset{edge/.style = {->}}

\node at (-1, 0.5) {$G_r$:};

\node[vertex] (v1) at (0, 1) {};
\node[vertex] (v2) at (1, 1) {};
\node[vertex] (v3) at (2, 1) {};
\node[dots] (v4) at (3, 1) {\(\cdots\)};
\node[vertex] (w1) at (0, 0) {\(s_r\)};
\node[vertex] (w2) at (1, 0) {};
\node[vertex] (w3) at (2, 0) {};
\node[dots] (w4) at (3, 0) {\(\cdots\)};

\draw[edge] (v1) to[bend left] node[right]{$\word{b}$} (w1);
\draw[edge] (v1) to node[above]{$\word{a}$} (v2);
\draw[edge] (v2) to node[right]{$\word{b}$} (w2);
\draw[edge] (v2) to node[above]{$\word{a}$} (v3);
\draw[edge] (v3) to node[right]{$\word{b}$} (w3);
\draw[edge] (v3) to node[above]{$\word{a}$} (v4);
\draw[edge] (w1) to[bend left] node[left]{$\word{a}$} (v1);
\draw[edge] (w2) to node[below]{$\word{b}$} (w1);
\draw[edge] (w3) to node[below]{$\word{b}$} (w2);
\draw[edge] (w4) to node[below]{$\word{b}$} (w3);

\node at (-1, -1.5) {$G_l$:};

\node[vertex] (v1) at (0, -1) {\(s_l\)};
\node[vertex] (v2) at (1, -1) {};
\node[vertex] (v3) at (2, -1) {};
\node[dots] (v4) at (3, -1) {\(\cdots\)};
\node[vertex] (w1) at (0, -2) {};
\node[vertex] (w2) at (1, -2) {};
\node[vertex] (w3) at (2, -2) {};
\node[dots] (w4) at (3, -2) {\(\cdots\)};

\draw[edge] (v1) to[bend left] node[right]{$\word{a}$} (w1);
\draw[edge] (v1) to node[above]{$\word{a}$} (v2);
\draw[edge] (v2) to node[right]{$\word{a}$} (w2);
\draw[edge] (v2) to node[above]{$\word{a}$} (v3);
\draw[edge] (v3) to node[right]{$\word{a}$} (w3);
\draw[edge] (v3) to node[above]{$\word{a}$} (v4);
\draw[edge] (w1) to[bend left] node[left]{$\word{b}$} (v1);
\draw[edge] (w2) to node[below]{$\word{b}$} (w1);
\draw[edge] (w3) to node[below]{$\word{b}$} (w2);
\draw[edge] (w4) to node[below]{$\word{b}$} (w3);

\end{tikzpicture}

\caption{Right-resolving and left-resolving labeled graphs for the $\word{a}^n\word{b}^n$-shift.}
\label{fig:context-free-shift-graph}
\end{figure}

Let $V_r$ and $V_l$ denote the set of vertices in $G_r$ and $G_l$ respectively.

\begin{lemma}\label{cf-shift-sync-vertex-correspondence}
Let $w$ be a synchronizing word in the $\word{a}^n\word{b}^n$-shift.  Then $w$ corresponds uniquely to an element of $V_l \times V_r$.
\end{lemma}
\begin{proof}
Let $p_l$ and $p_r$ be paths on $G_l$ and $G_r$ respectively which both represent the word $w$.  Since $w$ is synchronizing is contains $\word{ba}$, so $w = u\word{ba}v$ for some words $u$ and $v$.  Observe that the portion of $p_l$ representing $u\word{b}$ must terminate at $s_l$ on $G_l$, and likewise the portion of $p_r$ representing $\word{a}v$ must start at $s_r$ on $G_r$.  Since $G_l$ is left-resolving, there is a unique path representing $u\word{b}$ and terminating at $s_l$, so there is a unique vertex $x$ on $G_l$ where $p_l$ must initiate.  Similarly, since $G_r$ is right-resolving, there is a unique path representing $\word{a}v$ and starting at $s_r$ which must terminate at a unique vertex $y$ on $G_r$.  Hence $w$ corresponds uniquely to the pair $(x,y) \in V_l \times V_r$.
\end{proof}

In fact, two synchronizing word $w,v \in \lang(X)$ are context equivalent (see Definition \ref{def:shift-context-equivalence}) if and only if they correspond to the same element of $V_l \times V_r$.  Furthermore, we have the following result.

\begin{lemma}\label{lem:cf-shift-sync-vertex-bijection}
There is a bijective correspondence between context equivalence classes of synchronizing words and the set $V_l \times V_r$,
\end{lemma}
\begin{proof}
The correspondence in Lemma \ref{cf-shift-sync-vertex-correspondence} is injective since if two words $w,v \in \lang(X)$ correspond to the same element of $V_l \times V_r$ then it is easy to see $E(w) = E(v)$, hence they are in the same context equivalence class.  To show it is also surjective we use the irreducibility of the graphs $G_l$ and $G_r$.  Given $x \in V_l$ and $y \in V_r$, by irreducibility there are paths $p_l$ going from $x$ to $s_l$ and $p_r$ going from $s_r$ to $y$.  Let $w$ and $v$ be the labels of $p_l$ and $p_r$ respectively.  Since $s_l$ only has a single in-edge labeled $\word{b}$, and $s_r$ only a single out-edge labeled $\word{a}$, we must have $wv = w'\word{ba}v'$.  Hence $wv$ is synchronizing and corresponds to $(x,y) \in V_l \times V_r$.  Note that $wv$ is a valid word in the $\word{a}^n\word{b}^n$-shift since the follower set of the vertex $s_r$ is the same as the follower set of the vertex $s_l$.  Since $v$ is in the follower set of $s_r$ and $w$ terminates at $s_l$, it follows that there is a path in $G_l$ representing the word $wv$.
\end{proof}

Our goal is to show that $K_0\big(A(X_\word{ab}, \sigma)\big)$ is infinite rank.  We will now prove this.

\begin{theorem}\label{theorem:cf-shift-K-infinite-rank}
The abelian group $K_0\big(A(X_\word{ab}, \sigma)\big)$ is infinite rank.
\end{theorem}
\begin{proof}
Let $B^\lc(X_\word{ab}) = \left(\{ V_n \}_{n \geq -1}, \{ E_n \}_{n \geq -1} \right)$ be the Bratteli diagram constructed in Section \ref{sec:shift-bratteli-homoclinic}, and let $\widetilde{A}_n : \Z V_n \to \Z V_{n+1}$ denote the transition matrix induced by each $E_n$.  For $n \geq 0$, denote by $V^\sync_n \subseteq V_n$ the subset of vertices $v \in V_n$ such that $v = [w]$ with $w$ a synchronizing word in $\lang_{2n+1}(X_\word{ab})$.  By the bijective correspondence in Lemma \ref{lem:cf-shift-sync-vertex-bijection}, $V^\sync_n$ can be identified with a subset of $V_l \times V_r$.  Let \[ W_{2n + 1} = \Z(V^\sync_n) \subseteq \mathbb{Z}\left(V_l \times V_r\right) \cong \mathbb{Z}V_l \otimes \mathbb{Z}V_r \] denote the subspace spanned by $V^\sync_n$.

Denote by $V_l^n \subseteq V_l$ the set of vertices $\{v_1, \dots, v_n, w_1, \dots, w_n\}$ in $G_l$, and likewise for $V_r^n \subseteq V_r$.  Then $\Z V_l^n \otimes \Z V_r^n \subseteq \mathbb{Z}V_l \otimes \mathbb{Z}V_r$.  We get a $\Z$-linear map $A'_n : \Z V_l^n \otimes \Z V_r^n \to \Z V_l^{n+1} \otimes \Z V_r^{n+1}$ defined in terms of the matrix multiplication \[ A'_n(X) = A_n X A_n \] where $X \in \Z V_l^n \otimes \Z V_r^n$ is a $2n$ by $2n$ matrix and $A_n$ are the matrices defined in Section \ref{sec:cf-entropy}.

Recall that $V^\sync_n$ consists of equivalence classes of synchronizing words of length $2n + 1$ and that every synchronizing word must contain the word $\word{ba}$ as a subword.  Hence under the bijective correspondence in Lemma \ref{lem:cf-shift-sync-vertex-bijection}, any path representing $w$ on either of the graphs $G_r$ or $G_l$ must also be a path on $G_r$ or $G_l$ restricted to the vertices $V_l^n$ or $V_r^n$ respectively.  Hence $W_{2n+1} \subseteq \Z V_l^n \otimes \Z V_r^n$.  We will momentarily show the image of $W_{2n+1}$ under $A'_n$ satisfies \[ A'_n(W_{2n+1}) \subseteq W_{2n + 3} \subseteq \Z V_l^{n+1} \otimes \Z V_r^{n+1} \,. \]  Define $A^\sync_n : W_{2n+1} \to W_{2n+3}$ to be the map $A'_n$ restricted to $W_{2n + 1}$.

For each $n$, $A^\sync_n$ agrees with the transition matrix $\widetilde{A}_n$ restricted to $\Z V^\sync_n \subseteq \Z V_n$.  To see this, let $e_{ij} \in \Z V_l^n \otimes \Z V_r^n$ denote the matrix $x_i \otimes x_j$ where $1 \leq i,j \leq 2n$ and $x_i$ and $x_j$ are the $i$-th and $j$-th vertices in $V_l^n$ and $V_r^n$ respectively.  Hence the context equivalence class of a synchronizing word $w \in \lang_{2n + 1}(X)$ can be identified uniquely with some $e_{ij} \in \Z V_l^n \otimes \Z V_r^n$.  Then since $A_n$ are the adjacency matrices of the underlying unlabeled graphs of $G_r$ and $G_l$, one can see that \[ A'_n(e_{ij}) = A_n e_{ij} A_n = \sum_{\substack{(u,v) \in E(w) \\ |u| = |v| = 1}} [uwv] \,. \]  This inclusion $A^\sync_n(W_{2n+1}) \subseteq W_{2n + 3}$ follows from the fact that if $w$ is synchronizing and $(u,v) \in E(w)$ then $uwv$ is also synchronizing.

Each $A_n$ has determinant $\det(A_n) = -1$ (Lemma \ref{lem:cf-shift-transition-matrix-invertible}) and hence $A_n$ is invertible over $\Z$ for each $n$.  Hence $A'_n$ is injective onto its image: an inverse is given by the map $Y \to A_n^{-1} Y A_n^{-1}$.  Consequently $A'_n$ is full rank, and it follows that each $A^\sync_n$ is indeed full rank.  Note that $\rank(\widetilde{A}_n) \geq \rank(A^\sync_n)$ since $A^\sync_n$ is the restriction of $\widetilde{A}_n$ to $\Z V^\sync_n$.  Furthermore, the cardinality of $V^\sync_n$ increases without bound with respect to $n$ --- if this was not true it would contradict the bijective correspondence in Lemma \ref{lem:cf-shift-sync-vertex-bijection}.  In summary, since each $A^\sync_n$ is full rank, the rank of $\widetilde{A}_n$ increases without bound with respect to $n$.  Thus $K_0\big(A(X_\word{ab}, \sigma)\big)$ has infinite rank.
\end{proof}

We remark that the proof of Theorem \ref{theorem:cf-shift-K-infinite-rank} also shows that $K_0\big(\I_\sync(X_\word{ab}, \sigma)\big)$ has infinite rank.

\section{Example Related to the \emph{GICAR} Algebra}\label{sec:gicar}

Let $X \subseteq \{\word{a},\word{b},\word{c}\}^\mathbb{Z}$ be the closure of the set of bi-infinite paths on the following graph.

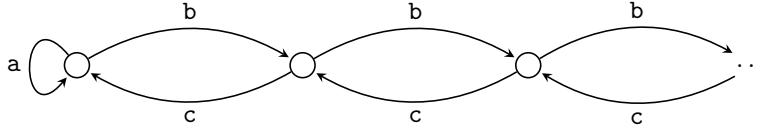
\begin{center}
\scalebox{1}{
\begin{tikzpicture}[
        > = stealth, 
        shorten > = 1pt, 
        auto,
        node distance = 3cm, 
        semithick, 
        scale = 0.5
    ]

    \tikzset{vertex/.style = {shape=circle,draw,minimum size=0.5em}}
    \tikzset{every loop/.style={looseness=10, in=230, out=130}}

    \node[vertex] (v1) {};
    \node[vertex] (v2) [right of=v1] {};
    \node[vertex] (v3) [right of=v2] {};
    \node (v4) [right of=v3] {\dots};

    \path[->] (v1) edge[loop left] node {$\word{a}$} (v1);
    \path[->] (v1) edge[bend left] node {$\word{b}$} (v2);
    \path[->] (v2) edge[bend left] node {$\word{b}$} (v3);
    \path[->] (v3) edge[bend left] node {$\word{b}$} (v4);
    \path[->] (v4) edge[bend left] node {$\word{c}$} (v3);
    \path[->] (v3) edge[bend left] node {$\word{c}$} (v2);
    \path[->] (v2) edge[bend left] node {$\word{c}$} (v1);

\end{tikzpicture}
}
\end{center}
This is a mixing synchronizing shift.  To see it is synchronizing, note that that a word $w \in \lang(X)$ is sychronizing if and only if $\word{a} \sqsubseteq w$.  It is also mixing by a similar argument as outlined in Example \ref{ex:lcu-lcs-sync-existence-counterexample}.  Furthermore, $X$ is not sofic since the words $\{ \word{ab}^n \mid n \geq 0 \}$ each have distinct follower sets, see \cite[Theorem 3.2.10]{lindmarcus}.

The following is the Bratteli diagram $B^\lcs_0(X,p)$ with respect to the periodic point $p = \overline{\word{a}} = \word{...aaa.aa...}$.  Note that the vertices of $B^\lcs_0(X,p)$ are in correspondence with equivalence classes of words $w \in E^+(\word{a})$, where $w \sim w'$ if and only if $|w| = |w'|$ and $E^+(\word{a}w) = E^+(\word{a}w')$ --- see Section \ref{sec:shift-bratteli-heteroclinic}.

\vspace{1em}

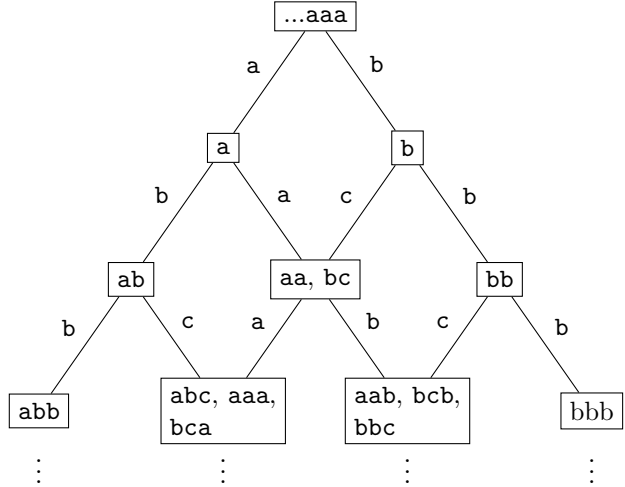
\begin{center}

\scalebox{1}{
\begin{tikzpicture}[scale=0.7]
\pgfmathsetmacro{\W}{2.5};
\pgfmathsetmacro{\H}{5};
\pgfmathsetmacro{\D}{1pt};

\node[shape=rectangle, draw=black] (S) at (0,-1*\H) {$\word{...aaa}$};

\node[shape=rectangle, draw=black] (V00) at (-0.7*\W,-1.5*\H) {$\word{a}$};
\node[shape=rectangle, draw=black] (V01) at (0.7*\W,-1.5*\H) {$\word{b}$};
    
    \path (S) edge node[above left] {$\word{a}$} (V00);
    \path (S) edge node[above right] {$\word{b}$} (V01);

\node[shape=rectangle, draw=black, align=left] (V10) at (-1.4*\W,-2*\H) {$\word{ab}$};
\node[shape=rectangle, draw=black, align=left] (V11) at (0,-2*\H) {$\word{aa}$, $\word{bc}$};
\node[shape=rectangle, draw=black, align=left] (V12) at (1.4*\W,-2*\H) {$\word{bb}$};

    \path (V00) edge node[above left] {$\word{b}$} (V10);
    \path (V00) edge node[above right] {$\word{a}$} (V11);
    \path (V01) edge node[above left] {$\word{c}$} (V11);
    \path (V01) edge node[above right] {$\word{b}$} (V12);
    
\node[shape=rectangle, draw=black, align=left] (V20) at (-2.1*\W,-2.5*\H) {$\word{abb}$};
\node[shape=rectangle, draw=black, align=left] (V21) at (-0.7*\W,-2.5*\H) {$\word{abc}$, $\word{aaa}$,\\$\word{bca}$};
\node[shape=rectangle, draw=black, align=left] (V22) at (0.7*\W,-2.5*\H) {$\word{aab}$, $\word{bcb}$,\\$\word{bbc}$};
\node[shape=rectangle, draw=black, align=left] (V23) at (2.1*\W,-2.5*\H) {$\word{bbb}$};

    \path (V10) edge node[above left] {$\word{b}$} (V20);
    \path (V10) edge node[above right] {$\word{c}$} (V21);
    \path (V11) edge node[above left] {$\word{a}$} (V21);
    \path (V11) edge node[above right] {$\word{b}$} (V22);
    \path (V12) edge node[above left] {$\word{c}$} (V22);
    \path (V12) edge node[above right] {$\word{b}$} (V23);
    
\node (E0) at (-2.1*\W,-2.7*\H) {$\vdots$};
\node (E1) at (-0.7*\W,-2.7*\H) {$\vdots$};
\node (E2) at (0.7*\W,-2.7*\H) {$\vdots$};
\node (E3) at (2.1*\W,-2.7*\H) {$\vdots$};
    
\end{tikzpicture}
}

\end{center}
It is interesting to note that the above Bratteli diagram corresponds to an AF-algebra called the \emph{GICAR} algebra, denoted $A_\text{GICAR}$, see \cite[Appendix]{renault1980groupoid}.  Renault proves that, \[ K_0(A_\text{GICAR}) \cong \Z[t] \] that is, the $K$-theory of $A_\text{GICAR}$ is the underlying abelian group of the ring of polynomials with integer coefficients.  This isomorphism enables a convenient description of the order structure: for $f \in K_0(A_\text{GICAR})$, $f > 0$ if and only if $f(t) > 0$ for all $t \in (0,1)$.

For us however, according Theorem \ref{theorem:shift-bratteli-lcu-lcs-isomorphism}, $S(X,\sigma,p)$ is isomorphic to the groupoid $C^\ast$-algebra of $G^\lcs(X,\sigma,p)$, which itself has the form \[ G^\lcs(X,\sigma,p) \cong \bigcup_{N \geq 0} G\left(B^\lcs_N(X,p)\right) \,. \]  Since $p = \word{...aaa.aaa...}$, $B^\lcs_N(X,p)$ is the same Bratteli diagram as $B^\lcs_0(X,p)$ for all $N$.  Hence $G^\lcs(X,\sigma,p)$ is an increasing union of the above Brattli diagram included into itself.  We have discussed how the vertices of $B^\lcs_0(X,p)$, and hence $B^\lcs_N(X,p)$, correspond to equivalence classes of words in $E^+(\word{a})$.  Let $[w;N]$ denote the vertex in $B^\lcs_N(X,p)$ corresponding to the equivalence class of the word $w$.  Then the inclusion of $B^\lcs_N(X,p)$ into $B^\lcs_{N+1}(X,p)$ sends the vertex $[w;N]$ to the vertex $[\word{a}w; N+1]$.

It is important to note that $G(B^\lcs_0(X,p))$ is not in general Morita equivalent to $G^\lcs(X,\sigma,p)$ --- we have only shown this is true for mixing sofic shifts, see Theorem \ref{theorem:shift-bratteli-lcu-lcs-morita-equivalence}.  In fact in this case, they are not Morita equivalent.  Furthermore, even though $X$ is mixing, we see in the proof of the following lemma that $S(X,\sigma,p)$ is not simple, which is in contrast to the mixing sofic case, see \cite[Theorem 6.10]{deeley2022}.

\begin{theorem}
The stable synchronizing algebra $S(X,\sigma,p)$ is not Morita equivalent to $A_\text{GICAR}$.
\end{theorem}
\begin{proof}
It is a fact that Morita equivalent $C^\ast$-algebras have isomorphic ideal structures, see \cite[Theorem 3.22]{raeburn1998morita}.  It is mentioned in \cite[Example III.5.5]{davidson1996c} that the ideals of $A_\text{GICAR}$ correspond to vertices in the Bratteli diagram $B^\lcs_0(X,p)$ above.  Hence ideals in $A_\text{GICAR}$ are in one-to-one correspondence with the vertices $[w;0]$ for $w \in E^+(\word{a})$.  The lattice structure on $\I(A_\text{GICAR})$ is given by the edge relation on the Bratteli diagram $B^\lcs_0(X,p)$.  Furthermore, each vertex $[w;N]$ corresponds to a proper ideal in $S(X,\sigma,p)$ since $G^0(B^\lcs_N(X,p)) \cong X^\text{u}(p, N)$ is an open invariant proper subset of $X^\text{u}(p)$, which is identified with the unit space of $G^\lcs(X,\sigma,p)$, see Theorem \ref{theorem:shift-bratteli-lcu-lcs-isomorphism} and \cite[Proposition 4.3.2]{sims2018etale}.  Note that we also identify $[w;N]$ with $[\word{a}w; N+1]$.

As can be read from the Bratteli diagram $B^\lcs_0(X,p)$: if $v \sqsubseteq w$ for $v,w \in E^+(\word{a})$, then $[w;N] \leq [v;N]$ in $\I(S(X,\sigma,p))$.  In particular, for $\epsilon \in E^+(\word{a})$ the empty word, $[\epsilon;N] = [\word{a};N+1] < [\epsilon;N+1]$ for all $N \geq 0$.  Hence we have that \[ [\epsilon; 0] < [\epsilon; 1] < [\epsilon; 2] < \dots \] is an infinite strictly ascending chain of ideals in $\I(S(X,\sigma,p))$ which is bounded below.  However, clearly no such chain can exist in $\I(A_\text{GICAR})$ since there are only finitely many vertices above a given vertex $[w;0]$ in $B^\lcs_0(X,p)$, and hence only a finite strictly ascending chain \[ [w_1w_2 \cdots w_n; 0] < [w_1w_2 \cdots w_{n-1}; 0] < \cdots < [w_1; 0] < [\epsilon; 0] \] is possible, where $n = |w|$ and $w = w_1w_2 \cdots w_n$.  Hence the respective lattices of ideals of $S(X,\sigma,p)$ and $A_\text{GICAR}$ cannot be isomorphic, and it follows that they cannot be Morita equivalent.
\end{proof}

Another aspect about this example we wish to highlight is that there is an intrinsic description of the quotient $A(X,\sigma)/\I_\sync \cong C^\ast(G^\lc_\text{non-sync}(X,\Sigma))$.

\begin{theorem}
Let $X \subseteq \{ \word{a}, \word{b}, \word{c} \}^\Z$ defined above.  Then $G^\lc_\text{non-sync}(X,\sigma)$ is defined by the following equivalence relation.  For $w \in \{ \word{b}, \word{c} \}^\ast$, let $d(w)$ be the integer \[ d(w) = |\{ i \mid 1 \leq i \leq |w|, w_i = \word{b} \}| - |\{ j \mid 1 \leq j \leq |w|, w_j = \word{c} \}| \,. \]  Define the equivalence relation $\sim$ on $\{ \word{b}, \word{c} \}^
\ast$ by \[ w \sim v \iff |w| = |v| \text{ and  } d(w) = d(v) \,. \]
\end{theorem}
\begin{proof}
First note that $X_\text{non-sync} = \{ \word{b}, \word{c} \}^\Z \subseteq X$, which follows from the fact that $w \in \lang(X)$ is synchronizing if and only if $\word{a} \sqsubseteq w$.  Hence the set of \emph{non-synchronizing words}, denoted $\lang_\text{non-sync}(X) = \lang(X) \setminus \lang_\sync(X)$, is exactly $\lang_\text{non-sync}(X) = \{ \word{b}, \word{c} \}^\ast$.

Next we claim that for $(x,y) \in G^\lc(X,\sigma)$, $(x,y)$ is in $G^\lc_\text{non-sync}(X,\sigma)$ if and only if there is an $N$ such that $x_{[-N,N]} \sim y_{[-N,N]}$ and $x_n = y_n$ for all $|n| > N$.  To see this, note that since $x \sim_\lc y$, there is an $N$ such that $E(x_{[-N,N]}) = E(y_{[-N,N]})$ and $x_n = y_n$ for $|n| > N$ --- see Lemma \ref{lem:shift-lc-conditions}.  Since $x_{[-N,N]}, y_{[-N,N]} \in \{ \word{b}, \word{c} \}^\ast$, then it is clear from the graph above that $E(x_{[-N,N]}) = E(y_{[-N,N]})$ if and only if $d(x_{[-N,N]}) = d(y_{[-N,N]})$.
\end{proof}
The equivalence relation $\sim$ on $\{ \word{b}, \word{c} \}^\ast$ is of interest in its own right, and also appears in the short exact sequence defined in Section \ref{sec:ses}.

\section{Synchronizing Shifts from Minimal Shifts}\label{sec:minimal}

The construction in this example demonstrates that there are synchronizing shifts $X$ such that $X_\text{non-sync}$ is minimal.  This construction is essentially the same as in \cite[Example 3.2]{petersen86}.

\begin{definition}
A shift space $M$ is called \emph{minimal} if the only $\sigma$-invariant closed subsets of $M$ are $\emptyset$ and $M$ itself.
\end{definition}
One can show that equivalently a shift space is minimal if the orbit $\{\sigma^n(x)\}_{n \in \Z}$ is dense for every $x \in M$.  In this section we will show how to build a synchronizing shift space from an element $m$ in a minimal shift space.  Note that one can also show that forward orbit, $\{\sigma^n(x)\}_{n \geq 0}$, is also dense in $M$ for any $x \in M$.

Let $\mathcal{A}$ be an alphabet and $M \subseteq \mathcal{A}^\Z$ a non-trivial (\ie not the orbit of a periodic point) minimal shift space.  Fix an element $m \in M$.  Consider the following infinite graph $G(m)$ with labels in $\widetilde{\mathcal{A}} = \mathcal{A} \sqcup \{ \word{e} \}$.

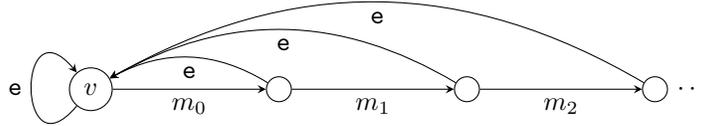
\vspace{1em}

\begin{center}

\begin{tikzpicture}[scale=2.5]

\tikzset{vertex/.style = {shape=circle, draw}}
\tikzset{dots/.style = {}}
\tikzset{edge/.style = {->,>=stealth}}
\tikzset{every loop/.style={looseness=6, in=130, out=230}}

\node[vertex] (v0) at (0, 0) {$v$};
\node[vertex] (v1) at (1, 0) {};
\node[vertex] (v2) at (2, 0) {};
\node[vertex] (v3) at (3, 0) {};
\node[dots] (v4) at (3.2, 0) {\(\cdots\)};

\draw[edge] (v0) to node[below]{$m_0$} (v1);
\draw[edge] (v1) to node[below]{$m_1$} (v2);
\draw[edge] (v2) to node[below]{$m_2$} (v3);
\draw[edge] (v1) to[bend right] node[below]{$\word{e}$} (v0);
\draw[edge] (v2) to[bend right] node[below]{$\word{e}$} (v0);
\draw[edge] (v3) to[bend right] node[below]{$\word{e}$} (v0);

\path[->,>=stealth] (v0) edge[loop left] node[left] {$\word{e}$} (v0);

\end{tikzpicture}

\end{center}

The closure of the set of labeled bi-infinite paths on the above graph defines a shift space $X(m)$ in the alphabet $\widetilde{\mathcal{A}}$.  The language of $X(m)$ is exactly the collection of words that are labels of finite paths on the above graph.

In particular, since the forward orbit of $m$ is dense in $M$, we have $\lang(M) \subseteq \lang(X(m))$.  To see this, let $x_m \in X(m)$ be defined by
 \[ (x_m)_i = \begin{cases} m_i & i \geq 0 \\ \word{e} & i < 0 \end{cases} \,. \]
Then $\overline{\{ \sigma^n(x_m) \}_{n \geq 0}} = M \subseteq X(m)$.  In other words, each $w \in \lang(M)$ occurs in $(x_m)_{[0,\infty)} = m_{[0, \infty)}$.  Hence $\lang(M) \subseteq \lang(X(m))$.

\begin{theorem}
The shift space $X(m)$ is synchronizing.  Furthermore, the set of non-synchronizing points in $X(m)$ is exactly the shift space $M$.
\end{theorem}
\begin{proof}
We will first show that $X(m)_\text{non-sync} \subseteq M$.  Suppose that $e \sqsubseteq w$, then we know that $w$ is represented by a path $q$ on $G(m)$ passing through the vertex $v$.  Since any such path must pass through $v$, it is then clear that $w$ is synchronizing by the technique used in Lemma \ref{cf-shift-sync-vertex-correspondence}.  Then, by minimality, since $X(m)_\text{non-sync} \subseteq M$ then either $X(m)_\text{non-sync} = M$ or $X(m)_\text{non-sync} = \emptyset$.  The latter cannot be true since it would imply $X$ is a shift of finite type, which is a contradiction --- see \cite[Example 3.2]{petersen86}.
\end{proof}

\end{document}